\documentclass{amsart}

\usepackage{amsfonts,amssymb,amsmath,amsthm}

\usepackage[all]{xy}
\usepackage{tikz}
\usepackage{tikz-cd}
\usepackage{caption}
\usepackage{subcaption}
\usepackage{graphicx}
\usepackage{hyperref}

\graphicspath{{Graphics/}}

\newtheorem{thm}{Theorem}[section]
\newtheorem{cor}[thm]{Corollary}
\newtheorem{lem}[thm]{Lemma}
\newtheorem{prop}[thm]{Proposition}
\newtheorem{conj}[thm]{Conjecture}

\theoremstyle{definition}
\newtheorem{defn}[thm]{Definition}
\newtheorem{exm}[thm]{Example}

\theoremstyle{remark}
\newtheorem{remark}[thm]{Remark}

\newcommand{\R}{\mathbb{R}}
\newcommand{\QC}{\mathcal{Q}}
\newcommand{\PC}{\mathcal{P}}
\newcommand{\OC}{\mathcal{O}}

\newcommand{\A}{\mathcal{A}}
\newcommand{\M}{\mathcal{M}}
\newcommand{\N}{\mathcal{N}}

\newcommand{\wPC}{\widehat{\mathcal{P}} }
\newcommand{\wDelta}{\widehat{\Delta}}
\newcommand{\wsigma}{\widehat{\sigma}}
\newcommand{\walpha}{\widehat{\alpha}}
\newcommand{\wgamma}{\widehat{\gamma}}
\newcommand{\reg}{\mbox{\emph{reg}}}
\newcommand{\syz}{\mbox{\emph{syz}}}
\newcommand{\row}{\mathfrak{r}}

\begin{document}

\title{Regularity of Mixed Spline Spaces}

\author{Michael DiPasquale}
\thanks{
Author supported by National Science Foundation grant DMS 0838434 ``EMSW21MCTP: Research Experience for Graduate Students.''}
\address{Department of Mathematics, University of Illinois, Urbana, Illinois, 61801}
\email{dipasqu1@illinois.edu}

\begin{abstract}
We derive bounds on the regularity of the algebra $C^\alpha(\PC)$ of mixed splines over a central polytopal complex $\PC\subset\R^3$.  As a consequence we bound the largest integer $d$ (the postulation number) for which the Hilbert polynomial $HP(C^\alpha(\PC),d)$ disagrees with the Hilbert function $HF(C^\alpha(\PC),d)=\dim C^\alpha(\PC)_d$.  The polynomial $HP(C^\alpha(\PC),d)$ has been computed in ~\cite{Assprimes}, building on~\cite{FatPoints,TSchenck08}.  Hence the regularity bounds obtained indicate when a known polynomial gives the correct dimension of the spline space $C^\alpha(\PC)_d$.  In the simplicial case with all smoothness parameters equal, we recover a bound originally due to Hong~\cite{HongDong} and Ibrahim and Schumaker~\cite{SuperSpline}.
\end{abstract}

\maketitle

\section{Introduction}
Let $\PC$ be a subdivision of a region in $\R^n$ by convex polytopes.  $C^r(\PC)$ denotes the set of piecewise polynomial functions (splines) on $\PC$ that are continuously differentiable of order $r$.  Splines are a fundamental tool in approximation theory and numerical analysis~\cite{Boor}; more recently they have also appeared in a geometric context, describing the equivariant cohomology ring of toric varieties \cite{Payne}.  Practical applications include surface modelling, computer-aided design, and computer graphics~\cite{Boor}.

One of the fundamental questions in spline theory is to determine the dimension of the space $C^r_d(\PC)$ of splines of degree at most $d$.  In the bivariate, simplicial case, these questions are studied by Alfeld-Schumaker in \cite{AS4r} and \cite{AS3r} using Bernstein-Bezier methods. A signature result in \cite{AS3r} is a formula for $\dim C^r_d(\Delta)$ when $d \ge 3r+1$ and $\Delta\subset\R^2$ is a generic simplicial complex.  For $\Delta\subset\R^2$ simplicial and nongeneric, Hong~\cite{HongDong} and Ibrahim-Schumaker~\cite{SuperSpline} derive a formula for $\dim C^r_d(\Delta)$ when $d\ge 3r+2$ as a byproduct of constructing local bases for these spaces.

An algebraic approach to the dimension question was pioneered by Billera in \cite{Homology} using homological and commutative algebra.  In~\cite{DimSeries}, Billera-Rose show that $C^r_d(\PC)\cong C^r(\wPC)_d$, the $d$th graded piece of the algebra $C^r(\wPC)$ of splines on the cone $\wPC$ over $\PC$.  The function $HF(C^r(\wPC),d)=\dim_\R C^r(\wPC)_d$ is known as the \textit{Hilbert function} of $C^r(\wPC)$ in commutative algebra, and a standard result is that the values of the Hilbert function eventually agree with the \textit{Hilbert polynomial} $HP(C^r(\wPC),d)$ of $C^r(\wPC)$.  An important invariant of $C^r(\wPC)$ is the \textit{postulation number} $\wp(C^r(\wPC))$, which is the largest integer $d$ so that $HP(C^r(\wPC),d)\neq HF(C^r(\wPC),d)$.  In this terminology the Alfeld-Schumaker result above could be viewed as a computation of $HP(C^r(\wDelta),d)$ plus the bound $\wp(C^r(\wDelta))\le 3r$.

The goal of this paper is to provide upper bounds on the postulation number $\wp(C^\alpha(\PC))$ for \textit{central} polytopal complexes $\PC\subset\R^{n+1}$, where $C^\alpha(\PC)$ is the algebra of \textit{mixed splines} over $\PC$.  A \textit{central} polytopal complex is one in which the intersection of all interior faces is nonempty; if $\PC$ is central then splines on $\PC$ are a graded algebra.  \textit{Mixed splines} are splines in which different smoothness conditions are imposed across codimension one faces.

The main reason for bounding $\wp(C^\alpha(\PC))$ is that the Hilbert polynomial of $C^\alpha(\PC)$ has been computed in situations where there are no known bounds on $\wp(C^\alpha(\PC))$, rendering these formulas impractical.  Currently, bounds which do not make heavy restrictions on the complex $\PC$ are known only in the simplicial case.  These bounds are recorded in Table~\ref{tbl:PostBounds}.  For particular types of complexes $\PC$, better and sometimes exact bounds are known for $\wp(C^r(\PC))$.
\begin{table}[tp]
\begin{tabular}{|lll|}
\multicolumn{3}{c}{Analytic Methods}\\
\hline
Bound & Context & Computed by\\
$\wp(C^r(\wDelta))\le 3r$ & \textit{generic} simplicial $\Delta\subset\R^2$& Alfeld-Schumaker ~\cite{AS3r} \\
$\wp(C^r(\wDelta))\le 3r+1$ & \textit{all} simplicial $\Delta\subset\R^2$ & Hong~\cite{HongDong}\\ & & Ibrahim-Schumaker~\cite{SuperSpline}\\
$\wp(C^1(\wDelta))\le 3$ & \textit{all} simplicial $\Delta\subset\R^2$ & Alfeld-Piper-Schumaker~\cite{APSr14}\\
$\wp(C^1(\wDelta))\le 7$ & \textit{generic} simplicial $\Delta\subset\R^3$ & Alfeld-Schumaker-Whiteley ~\cite{ASWTet}\\
\hline
\end{tabular}

\begin{tabular}{|lll|}
\multicolumn{3}{c}{Homological Methods}\\
\hline
Bound & Context & Computed by\\
$\wp(C^r(\wDelta))\le 4r$ & \textit{all} simplicial $\Delta\subset\R^2$ & Mourrain-Villamizar~\cite{D2}\\
$\wp(C^1(\wDelta))\le 1$ & \textit{generic} simplicial $\Delta\subset\R^2$ & Billera~\cite{Homology}\\
\hline
\end{tabular}
\caption{Bounds on $\wp(C^r(\wDelta))$}\label{tbl:PostBounds}
\end{table}
In contrast, the Hilbert polynomial $HP(C^\alpha(\PC),d)$ has been computed for \textit{all} central polytopal complexes $\PC\subset\R^3$.  This is done in the simplicial case with mixed smoothness by Schenck-Geramita~\cite{FatPoints}, in the polytopal case with uniform smoothness by Schenck-McDonald~\cite{TSchenck08}, and in the polytopal case with mixed smoothness and boundary conditions, by the author~\cite{Assprimes}.  In this paper we provide the first bound on $\wp(C^\alpha(\PC))$ for all central polytopal complexes $\PC\subset\R^3$.  Specifically, given \textit{smoothness parameters} $\alpha(\tau)$ associated to each codimension one face $\tau\in\PC$, our first result is the following.

\medskip

\noindent\textbf{Theorem~\ref{thm:main2}}
Let $\PC\subset\R^3$ be a central, pure, hereditary three-dimensional polytopal complex.  Set
\[
e(\PC)=\max\limits_{\tau\in\PC^0_2}\{\sum\limits_{\gamma\in (\mbox{st}(\tau))_2} (\alpha(\gamma)+1)\},
\]
where $\mbox{st}(\tau)$ denotes the star of $\tau$ and $(\mbox{st}(\tau))_2$ denotes the $2$-faces of $\mbox{st}(\tau)$.  Then
\[
\wp(C^\alpha(\PC))\le e(\PC)-3.
\]
In particular, $HP(C^\alpha(\PC),d)=\dim_\R C^\alpha(\PC)_d$ for $d\ge e(\PC)-2$.

\medskip

From an algebraic perspective, another reason for bounding $\wp(C^\alpha(\PC))$ is that almost all existing bounds, including most in Table~\ref{tbl:PostBounds}, have been computed using analytic techniques.  There are a few instances where algebraic techniques are applied to bound $\wp(C^\alpha(\PC))$.  In~\cite{Homology}, Billera proves $\wp(C^1(\wDelta))\le 1$ for \textit{generic} simplicial complexes (this result relies on a computation of Whiteley~\cite{WhiteleyM}).  The most general bound produced by homological techniques to date is by Mourrain-Villamizar~\cite{D2}; building on work of Schenck-Stillman~\cite{LCoho} they prove that $\wp(C^r(\wDelta))\le 4r$ for $\Delta$ a planar simplicial complex, recovering an earlier result of Alfeld-Schumaker~\cite{AS4r}.  Our second result is the following.

\medskip

\noindent\textbf{Theorem~\ref{thm:main3}}
Let $\Delta\subset\R^3$ be a central, pure, hereditary three-dimensional simplicial complex.  For a $2$-face $\tau\in\Delta^0_2$, set
\[
M(\tau)=(\alpha(\tau)+1)+\max\{(\alpha(\gamma_1)+1)+(\alpha(\gamma_2)+1)|\gamma_1\neq\gamma_2 \in(\mbox{st}(\tau))_2 \}.
\]
Then
\[
\wp(C^\alpha(\Delta))\le \max\limits_{\tau\in\Delta^0_2}\{M(\tau)\}-2.
\]
In particular, $HP(C^\alpha(\Delta),d)=\dim_\R C^r(\Delta)_d$ for $d\ge \max\limits_{\tau\in\Delta^0_2}\{M(\tau)\}-1$.

\medskip

Setting $\alpha(\tau)=r$ for all $\tau\in\Delta^0_2$, we recover that $HP(C^r(\wDelta),d)=\dim C^r(\wDelta)_d$ for $d\ge 3r+2$.  This was originally proved via constructing local bases by Hong~\cite{HongDong} and Ibrahim-Schumaker~\cite{SuperSpline}, and is the best bound valid for all planar simplicial complexes recorded in Table~\ref{tbl:PostBounds}.


A key tool we use to prove these results is the \textit{Castelnuovo-Mumford regularity} of $C^\alpha(\PC)$, denoted $\reg(C^\alpha(\PC))$.  The relationship between $\reg(C^\alpha(\PC))$ and $\wp(C^\alpha(\PC))$ is discussed in detail in \S~\ref{sec:regularity}.  This invariant is also used by Schenck-Stiller in~\cite{CohVan}.  Our particular way of using regularity is inspired by an observation used in the Gruson-Lazarsfeld-Peskine theorem bounding the regularity of curves in projective space.  In the context of splines this observation is roughly that, if we are lucky, we can bound $\reg(C^\alpha(\PC))$ by the regularity of a `bad' approximation.  This statement is made precise in Proposition~\ref{prop:RegDepth} and Theorem~\ref{thm:main0}.  We take as our approximation certain locally-supported subalgebras of splines introduced in~\cite{LS}.  This could be viewed as an algebraic analogue of locally-supported bases used in~\cite{HongDong,SuperSpline}.

The paper is organized as follows.  In \S~\ref{sec:prelim} we give some background on the spline algebra $C^\alpha(\PC)$, in particular the algebraic approach pioneered by Billera~\cite{Homology} and Billera and Rose~\cite{DimSeries}.  We recall the construction of lattice-supported splines $LS^{\alpha,k}(\PC)$ introduced in~\cite{LS}.  These will provide approximations to $C^\alpha(\PC)$.  In \S~\ref{sec:Int} and \S~\ref{sec:splineint} we fit lattice-supported splines into a Cech-like complex.  In \S~\ref{sec:regularity} we recall the definition of the regularity of a graded module and prove Proposition~\ref{prop:RegDepth}, which is our main tool for bounding regularity.  In \S~\ref{sec:LowPD} we prove our main results for bounding regularity of spline modules of low projective dimension and prove Theorem~\ref{thm:main2} bounding the regularity of $C^\alpha(\PC)$ where $\PC\subset\R^3$ is a central polytopal complex.  In \S~\ref{sec:SStar} we build on work of Tohaneanu-Minac ~\cite{Stefan} and prove the more precise regularity estimate for central simplicial complexes $\Delta\subset\R^3$ in Theorem~\ref{thm:main3}.  We close in \S~\ref{sec:conjectures} with conjectured regularity bounds generalizing those derived in this paper.  The two following examples illustrate our results.

\subsection{Examples}
Let $\PC\subset\R^2$ be a subdivision of a topological $2$-disk by convex polytopes.  $C^r(\PC)$ is the algebra of $r$-splines on $\PC$, where $\alpha(\tau)=r$ for every interior edge and $\alpha(\tau)=-1$ for every boundary edge.
By Corollary~3.14 of \cite{TSchenck08}, the Hilbert polynomial of $C^r(\wPC)$ is
\begin{eqnarray}\label{eqn:HPolytopal}
HP(C^r(\wPC),d)= & \frac{f_2}{2}d^2 +\frac{3f_2-2(r+1)f^0_1}{2}d+f_2+\left(\binom{r}{2}-1\right)f^0_1+\sum_j c_j,
\end{eqnarray}
where $f_i,f^0_i$ are the number of $i$-faces and interior $i$-faces of $\PC$, $r$ is the smoothness parameter, and the constants $c_j$ record the dimension of certain vector spaces coming from ideals of powers of linear forms.

\begin{exm}\label{ex:HPolytopal}
The complex $\QC$ in Figure~\ref{fig:HPolytopal} has $f_2=4,f^0_1=6,f^0_3=3$.  It is shown in \S~4 of \cite{TSchenck08} that there are $4$ constants $c_j$ in the formula~\eqref{eqn:HPolytopal}, and they are all equal to the constant
\[
\binom{r+2}{2}+\left\lceil \frac{r+1}{2}\right\rceil \left( r-\left\lceil \frac{r+1}{2}\right\rceil \right)
\]
Hence by equation~\eqref{eqn:HPolytopal},
\begin{eqnarray}\label{eqn:HPQC1}
HP(C^r(\widehat{\QC}),d)= & 2d^2-6rd+6\binom{r}{2}-2+4\left(\binom{r+2}{2}+\left\lceil \frac{r+1}{2}\right\rceil \left( r-\left\lceil \frac{r+1}{2}\right\rceil \right)\right)
\end{eqnarray}

\begin{figure}[htp]

\centering
\begin{tikzpicture}
\draw (-1,-1) -- (1,-1) --(0,1) node [right] {$(0,1)$} --(0,3) node [above] {$(0,3)$}--(-3,-2) node [below] {$(-3,-2)$} -- (3,-2) node [below] {$(3,-2)$};
\node at (-1.6,-.8) {$(-1,-1)$};
\node at (1.5,-.8) {$(1,-1)$};
\draw (1,-1)--(3,-2) -- (0,3);
\draw (-3,-2)--(-1,-1)--(0,1);
\draw [fill] (-1,-1) circle [radius=.05];
\draw [fill] (-3,-2) circle [radius=.05];
\draw [fill] (1,-1) circle [radius=.05];
\draw [fill] (3,-2) circle [radius=.05];
\draw [fill] (0,1) circle [radius=.05];
\draw [fill] (0,3) circle [radius=.05];
\end{tikzpicture}
\caption{$\QC$}\label{fig:HPolytopal}
\end{figure}

By Theorem~\ref{thm:main2}, $\wp(C^r(\widehat{\QC}))\le e(\QC)-3$, where 
\[
e(\QC)=\max\limits_{\tau\in\PC^0_2}\{\sum\limits_{\gamma\in (\mbox{st}(\tau))_2} (\alpha(\gamma)+1)\}.
\]
The star of each interior edge of $\QC$ has $5$ edges which are interior.  So $e(\QC)=5(r+1)$ and the Hilbert function $HF(C^r(\widehat{\QC}),d)$ agrees with the Hilbert polynomial $HP(C^r(\widehat{\QC}),d)$ above for $d\ge 5(r+1)-2$.  Computations in Macaulay2~\cite{M2} suggest that $\wp (C^r(\widehat{\QC}))\le 2(r+1)-1$ (in fact the behavior is the same as Example~\ref{ex:SchlegelCube} in \S~\ref{sec:examples}), indicating that there is room for improvement in Theorem~\ref{thm:main2}.
\end{exm}

In \cite[Theorem~4.3]{FatPoints}, Geramita and Schenck give a formula for $HP(C^\alpha(\wDelta),d)$, where $\Delta$ is a planar simplicial complex and $\alpha$ is the vector of smoothness parameters associated to codimension one faces.

\begin{exm}\label{ex:HSimplicial}
Triangulate the polytopal complex $\QC$ in Example~\ref{ex:HPolytopal} to obtain the simplicial complex $\Delta$ below, with $f_2=7,f^0_1=9,$ and $f^0_0=3$.  Take smoothness parameters $\alpha(\tau)=2$ on the edges of the center triangle and $\alpha(\tau)=3$ on the six edges which connect interior vertices to boundary vertices.
\begin{figure}[htp]
\centering
\begin{tikzpicture}[scale=.6]
\draw (-1,-1) -- (1,-1) --(0,1) --(0,3)--(-3,-2) -- (3,-2);
\draw (1,-1)--(3,-2) -- (0,3);
\draw (-3,-2)--(-1,-1)--(0,1);
\draw (-1,-1)--(3,-2);
\draw (-3,-2)--(0,1);
\draw (0,3)--(1,-1);
\draw [fill] (-1,-1) circle [radius=.05];
\draw [fill] (-3,-2) circle [radius=.05];
\draw [fill] (1,-1) circle [radius=.05];
\draw [fill] (3,-2) circle [radius=.05];
\draw [fill] (0,1) circle [radius=.05];
\draw [fill] (0,3) circle [radius=.05];
\end{tikzpicture}
\caption{$\Delta$}\label{fig:HSimplicial}
\end{figure}
In Example~4.5 of \cite{FatPoints}, Schenck and Geramita compute
\[
HP(C^\alpha(\wDelta),d)=\binom{d+2}{2}-3\binom{d-1}{2}+3\binom{d-2}{2}+6\binom{d-3}{2}.
\]
By Theorem~\ref{thm:main3}, $\wp(C^\alpha(\wDelta))\le\max\{ M(\tau)|\tau\in\Delta^0_1\}-2$, where $M(\tau)=\alpha(\tau)+1+\max\{\alpha(\gamma_1)+1+\alpha(\gamma_2)+1|\gamma_1\neq\gamma_2 \in (\mbox{st}(\tau))_1\}$.  This yields $\wp (C^\alpha(\wDelta))\le 10$, so the polynomial above gives the correct dimension of $C^\alpha_d(\Delta)$ for $d\ge 11$.  Macaulay2 gives $\wp (C^\alpha(\wDelta))=5$, so the formula is actually correct for $d\ge 6$.
\end{exm}

\section{Splines and Lattice-Supported Splines}\label{sec:prelim}
We begin with some preliminary notions.  A \textit{polytopal complex} $\PC\subset\R^n$ is a finite set of convex polytopes (called \textit{faces} of $\PC$) in $\R^n$ such that
\begin{itemize}
\item If $\gamma\in\PC$, then all faces of $\gamma$ are in $\PC$.
\item If $\gamma,\tau\in\PC$ then $\gamma\cap\tau$ is a face of both $\gamma$ and $\tau$ (possibly empty).
\end{itemize}
The \textit{dimension} of $\PC$ is the greatest dimension of a face of $\PC$. 
The faces of $\PC$ are ordered via inclusion; a maximal face of $\PC$ is 
called a \textit{facet} of $\PC$, and $\PC$ is said to be \textit{pure} if all facets are equidimensional.  $|\PC|$ denotes the underlying space of $\PC$. $\PC_i$ and $\PC^0_i$ denote the set of $i$-faces and the set of interior $i$-faces, respectively, while $f_i=|\PC_i|$ and $f^0_i=|\PC^0_i|$.  In the case that all facets of $\PC$ are simplices, $\PC$ is a simplicial complex and will be denoted by $\Delta$.  The boundary of $\PC$, denoted $\partial\PC$, is a polytopal complex of dimension $n-1$, and is pure if $\PC$ is pure.

Given a complex $\PC$ and a face $\gamma\in\PC$, the \textit{star} of $\gamma$ in $\PC$, denoted $\mbox{st}_\PC  (\gamma)$, is defined by
\[
\mbox{st}_\PC (\gamma):=\{\psi\in\PC|\exists\sigma\in\PC, \psi\in\sigma,\gamma\in\sigma \}.
\]
This is the smallest subcomplex of $\PC$ which contains all faces which contain $\gamma$.  If the complex $\PC$ is understood we will write $\mbox{st}(\gamma)$.

For $\PC\subset\R^n$, we define $G(\PC)$ to be the graph with a vertex for every facet (element of $\PC_n$); two vertices are joined by an edge iff the corresponding facets $\sigma$ and $\sigma'$ satisfy $\sigma\cap\sigma'\in\PC_{n-1}$.  $\PC$ is said to be \textit{hereditary} if $G(\mbox{st}_\PC(\gamma))$ is connected for every nonempty $\gamma\in\PC$.  Throughout this paper, $\PC\subset\R^n$ is assumed to be a pure, $n$-dimensional, hereditary polytopal complex.

Let $R=\R[x_1,\ldots,x_n]$ be the polynomial ring in $n$ variables, and $S=\R[x_0,\ldots$ $,x_n]$ the polynomial ring in $n+1$ variables.  We will typically use $R$ in inhomogeneous and $S$ in homogeneous situations.

We now recall the definition of the ring of splines $C^r(\PC)$.  For $U\subset \R^n$, let $C^r(U)$ denote the set of functions $F:U\rightarrow \R$ continuously differentiable of order $r$.  For $F:|\PC|\rightarrow \R$ a function and $\sigma\in\PC_n$, $F_\sigma$ denotes the restriction of $F$ to $\sigma$.  The module $C^r(\PC)$ of piecewise polynomials continuously differentiable of order $r$ on $\PC$ is defined by
\[
C^r(\PC):=\{F\in C^r(|\PC|)|F_\sigma\in R \mbox{ for every } \sigma\in\PC_n\}
\]
The polynomial ring $R$ includes in $C^r(\PC)$ as globally polynomial functions (these are the \textit{trivial} splines); this makes $C^r(\PC)$ an $R$-algebra via pointwise multiplication.

If $\PC$ is a hereditary complex, the global $C^r$ condition can be expressed as a differentiability condition across internal faces of codimension one.  For a codimension one face $\tau\in\PC^0_{n-1}$, let $l_\tau$ denote a choice of affine form (unique up to scaling) which vanishes on $\tau$.  Then a function $F:|\PC|\rightarrow \R$ which restricts to a polynomial on each facet is in $C^r(\PC)$ iff $l^{r+1}_\tau|(F_{\sigma_1}-F_{\sigma_2})$ for every pair of facets $\sigma_1,\sigma_2$ which intersect in a codimension one face $\tau$~\cite{DimSeries}.

In \cite{FatPoints}, Schenck and Geramita study the dimension of \textit{mixed spline} spaces, in which the order of differentiability across codimension one faces varies.  Specifically, let $\alpha=(\alpha(\tau)|\tau\in\PC_{n-1})$ be a list of \textit{smoothness parameters} $\alpha(\tau)$ associated to each codimension one face.  We require that $\alpha(\tau)\ge 0$ for $\tau\in\PC^0_{n-1}$ and $\alpha(\tau)\ge -1$ for $\tau\in(\partial\PC)_{n-1}$.  Then the algebra $C^\alpha(\PC)$ of mixed splines on $\PC$ is defined as the set of splines $F\in C^0(\PC)$ satisfying
\begin{itemize}
\item $l^{\alpha(\tau)+1}_\tau|(F_{\sigma_1}-F_{\sigma_2})$ for $\tau\in\PC^0_{n-1}$ with $\sigma_1\cap\sigma_2=\tau$
\item $l^{\alpha(\tau)+1}_\tau|F_\sigma$ for $\tau\in(\partial\PC)_{n-1}$ with $\tau\in(\partial\sigma)_{n-1}$
\end{itemize}
For hereditary complexes, the usual ring of splines $C^r(\PC)$ is recovered by setting $\alpha(\tau)=r$ for every $\tau\in\PC^0_{n-1}$ and $\alpha(\tau)=-1$ for every $\tau\in\PC_{n-1}$.  The following variant of~\cite[Proposition~ 4.3]{DimSeries} encodes the mixed spline conditions as a matrix.
\begin{lem}\label{lem:seq1}
If $\PC$ is hereditary and $\alpha=(\alpha(\tau)|\tau\in\PC_{n-1})$, $C^\alpha(\PC)$ fits into the graded exact sequence
\[
\begin{array}{c}
0 \rightarrow C^\alpha(\PC) \rightarrow R^{f_n} \oplus \left(\bigoplus\limits_{\tau\in\PC_{n-1} } R(-\alpha(\tau)-1)\right) \xrightarrow{\phi} R^{f_{n-1}} \rightarrow C \rightarrow 0\\
\textup{where } \phi= \begin{pmatrix}
& \vline & l^{\alpha(\tau_1)+1}_{\tau_1} & & \\
\delta_n & \vline  & & \ddots & \\
 & \vline & & & l^{\alpha(\tau_k)+1}_{\tau_k}
\end{pmatrix},
\end{array}
\]
$k=f_{n-1}$, $C=\textup{coker } \phi$ and the matrix $\delta_n$ is the top dimensional cellular boundary map of $\PC$.
\end{lem}
Since our results apply in the context of mixed splines, we will use these throughout the paper.

Let $R_{\le d}$ and $R_d$ be the set of polynomials $f\in R$ of degree $\le d$ and degree $d$, respectively.  For $\PC\subset\R^n$ we have a filtration of $C^\alpha(\PC)$ by $\R$-vector spaces
\[
C^\alpha_d(\PC):=\{F\in C^\alpha(\PC)|F_{\sigma}\in R_{\le d} \mbox{ for all facets }\sigma\in\PC_n\}.
\]
A polytopal complex $\PC$ is called a \textit{central} complex if all interior codimension one faces share a common face.  We will always assume the origin $\mathbf{0}\in\R^n$ is contained in this common face.  For central complexes
we make the assumption that $\alpha(\tau)=-1$ for all codimension one faces $\tau\in\PC_{n-1}$ so that $\mathbf{0}\notin\mbox{aff}(\tau)$.  Then the diagonal portion of the matrix $\phi$ in Lemma~\ref{lem:seq1} consists of forms of degree $\alpha(\tau)+1$ and $C^\alpha(\PC)$ is graded.  The graded pieces are the vector spaces
\[
C^\alpha(\PC)_d:=\{F\in C^\alpha(\PC)|F_{\sigma}\in R_d \mbox{ for all facets }\sigma\in\PC_n\}.
\]
Given $\PC\subset\R^n$, the cone $\wPC\subset\R^{n+1}$ over $\PC$ is formed by taking the join of $\mathbf{0}\in\R^{n+1}$ with $i(\PC)$, where $i:\R^n\rightarrow\R^{n+1}$ is defined by $i(a_1,\ldots,a_n)=(1,a_1,\ldots,a_n)$.  This is clearly a central complex.  If $\PC$ comes with smoothness parameters $\alpha$, extend these to smoothness parameters $\walpha$ on $\wPC$ by assigning
\begin{itemize}
\item $\walpha(\tau')=\alpha(\tau)$ for $\tau'\in\wPC_n$ which is a cone over $\tau\in\PC_{n-1}$ and
\item $\walpha(\tau')=-1$ for $\tau'\in\wPC_n$ so that $\mathbf{0}\notin\mbox{aff}(\tau')$
\end{itemize}
Since this extension is natural we will abuse notation and drop the hat on $\alpha$, denoting $C^{\walpha}(\wPC)$ by $C^{\alpha}(\wPC)$.  In practice one computes the algebra $C^{\walpha}(\wPC)$ by replacing the polynomial ring $R$ by $S$ in Lemma~\ref{lem:seq1} and homogenizing the entries of the matrix $\phi$ used to compute $C^\alpha(\PC)$. The following lemma is proved the same way as Theorem~2.6 of \cite{DimSeries}.

\begin{lem}\label{lem:splinehom}
Let $\PC\subset\R^n$ be a polytopal complex with smoothness parameters $\alpha$.  Then $C^\alpha_d(\PC)\cong C^{\walpha}(\wPC)_d$ as $\R$-vector spaces.
\end{lem}

\subsection{Lattice-Supported Splines}\label{sec:Lat}
In ~\cite{LS} certain subalgebras $LS^{r,k}(\PC)\subset C^r(\PC)$ are constructed as approximations to $C^r(\PC)$.  This construction carries over directly to mixed splines; we will denote the corresponding submodules by $LS^{\alpha,k}(\PC)$.  We briefly summarize the construction.  For a pure $n$-dimensional subcomplex $\QC\subset\PC$, not necessarily hereditary, define
\[
C^\alpha_\QC(\PC):=\{F\in C^\alpha(\PC)|F_\sigma=0 \mbox{ for all }\sigma\in\PC_n \setminus \QC_n \}.
\]
Let $\PC^{-1}\subset\partial\PC$ denote the set of faces of $\PC$ which are contained in a codimension one face $\tau$ so that $\alpha(\tau)=-1$; this is a subcomplex of $\partial\PC$.

\begin{defn}
Let $\PC\subset\R^n$ be a polytopal complex and $\alpha$ a list of smoothness parameters.
\begin{enumerate}
\item For $\tau\in\PC$ a face, $\mbox{aff}(\tau)$ denotes the affine span of $\tau$.
\item $\A(\PC,\PC^{-1})$ denotes the hyperplane arrangement $\bigcup\limits_{\substack{\tau\in\PC_{n-1}\\ \alpha(\tau)\ge 0}}\mbox{aff}(\tau)$.
\item $L_{\PC,\PC^{-1}}$ denotes the intersection semi-lattice $L(\A(\PC,\PC^{-1}))$ of $\A(\PC,\PC^{-1})$.
\end{enumerate}
\end{defn}
The elements $W\in L(\PC,\PC^{-1})$ are called \textit{flats}.  These consist of the whole space $\R^n$, the hyperplanes $\{\mbox{aff}(\tau)|\alpha(\tau)\ge 0\}$, and all nonempty intersections of these, ordered with respect to reverse inclusion.  The \textit{rank} of a flat $W$, denoted $\mbox{rk}(W)$, is its codimension as a vector space.

To each flat $W\in L(\PC,\PC^{-1})$ we associate a \textit{lattice complex} $\PC_W$ as follows.  Form a graph $G_W(\PC)$ whose vertices correspond to facets which have a codimension one face $\tau$ so that $W\subseteq \mbox{aff}(\tau)$.  Connect two vertices corresponding to facets $\sigma_1,\sigma_2$ if $\sigma_1\cap\sigma_2$ is a codimension one face of both and $W\subseteq\mbox{aff}(\sigma_1\cap\sigma_2)$.  Each connected component $G^i_W(\PC)$ of $G_W(\PC)$ is the dual graph of a unique subcomplex $\PC^i_W$.  The lattice complex $\PC_W$ is defined as the disjoint union of these $\PC^i_W$, which we call components of $\PC_W$.  Then define
\[
C^\alpha_W(\PC):=\sum_i C^\alpha_{\PC^i_W}(\PC),
\]
the submodule generated by splines which vanish outside a component of $\PC_W$.  Then $LS^{\alpha,k}(\PC)$ is defined by
\[
LS^{\alpha,k}(\PC):=\sum\limits_{\substack{W\in L_{\wPC,\wPC^{-1}}\\ 0\leq\mbox{\emph{rk}}(W)\leq k}} C^\alpha_W(\PC).
\]
It is equivalent to let the sum in the definition of $LS^{\alpha,k}(\PC)$ run across maximal components (with respect to inclusion) occuring among lattice complexes $\PC_W$ with the rank of $W$ at most $k$.  To make this more precise, let $\Gamma^k_\PC$ be the poset of components of lattice complexes $\PC_W$ with $\mbox{rk}(W)\le k$, ordered with respect to inclusion.  Let $\Gamma^{k,\textup{max}}_\PC$ be the set of maximal subcomplexes appearing in $\Gamma^k_\PC$.  Then we have
\begin{prop}~\cite[Proposition~4.9]{LS}\label{prop:top}
\[
LS^{\alpha,k}(\PC)=\sum\limits_{\QC\in\Gamma^{k,\textup{max}}_\PC} C^\alpha_\QC(\PC).
\]
\end{prop}

Since we will use this construction primarily in the cases $k=0$ and $k=1$, we describe $LS^{\alpha,0}(\PC)$ and $LS^{\alpha,1}$ precisely.  If $\gamma$ is a face of $\PC$ of some dimension, we use $C^\alpha_\gamma(\PC)$ and $C^\alpha_{\mbox{st}(\gamma)}(\PC)$ interchangeably to denote the subring of splines which vanish outside of the star of $\gamma$, as long as no confusion results.  So $C^\alpha_{\sigma}(\PC)$ for $\sigma\in\PC_n$ denotes the subring of splines supported on a single facet, $C^\alpha_\tau (\PC)$ for $\tau\in\PC^0_{n-1}$ denotes the ring of splines supported on the two facets of $\mbox{st}(\tau)$, etc.
\begin{cor}\label{cor:LowkDescription}
Let $\PC\subset\R^n$ be a polytopal complex.  Then 
\[
\begin{array}{rl}
LS^{\alpha,0}(\PC)= & \sum\limits_{\sigma\in\PC_n} C^\alpha_\sigma(\PC)\\
LS^{\alpha,1}(\PC)= & \sum\limits_{\tau\in\PC^0_{n-1}} C^\alpha_\tau(\PC)
\end{array}
\]
\end{cor}
\begin{proof}
For $k=0$, the only flat $W\in L(\PC,\PC^{-1})$ of rank zero is the whole space $\R^n$.  The corresponding lattice complex $\PC_{\R^n}$ is the disjoint union of the facets of $\PC$. Hence
\[
LS^{\alpha,0}(\PC)=C^\alpha_{\R^n}(\PC)=\sum\limits_{\sigma\in\PC_n} C^\alpha_\sigma(\PC).
\]
For $k=1$, the flats $W\in L(\PC,\PC^{-1})$ of rank one are precisely the hyperplanes $\mbox{aff}(\tau)$ with $\alpha(\tau)\ge 0$, where $\tau\in\PC_{n-1}$.  The components of the lattice complex $\PC_{\mbox{aff}(\tau)}$ are the complexes $\mbox{st}(\gamma)$ for all $\gamma$ satisfying $\mbox{aff}(\gamma)=\mbox{aff}(\tau)$.  If $\gamma\in\PC^0_{n-1}$, then $\mbox{st}(\gamma)$ consists of two facets and all their faces; otherwise $\gamma\in(\partial\PC)_{n-1}$ and $\mbox{st}(\gamma)$ consists of a single facet of $\PC$ and all its faces.  However, as long as $\PC$ is hereditary and has more than one facet, every facet $\sigma\in\PC_n$ has a codimension one face $\gamma$ which is interior.  Hence $\sigma\subset\mbox{st}(\gamma)$.  It follows that $\Gamma^{1,\textup{max}}_\PC$ consists of stars of interior codimension one faces of $\PC$.  By Proposition~\ref{prop:top} we have
\[
LS^{\alpha,1}(\PC)=\sum\limits_{\tau\in\PC^0_{n-1}} C^\alpha_\tau(\PC)
\]
\end{proof}

Theorem~4.3 of~\cite{LS} makes precise the sense in which $LS^{r,k}(\PC)$ is an approximation to $C^r(\PC)$.  This result and its proof extend directly to mixed splines, so we state the result in this context.

\begin{thm}\label{thm:latticegens}
Let $\PC\subset\R^n$ be a polytopal complex.  Then $LS^{\alpha,k}(\PC)$ fits into a short exact sequence
\[
0\rightarrow LS^{\alpha,k}(\PC) \rightarrow C^\alpha(\PC) \rightarrow C \rightarrow 0
\]
where $C$ has codimension $\geq k+1$ and the primes in the support of $C$ with codimension $k+1$ are contained in the set $\{I(W)|W\in L(\PC,\PC^{-1})\mbox{ \emph{and} }\mbox{\emph{rk}}(W)=k+1\}$.
\end{thm}

To use the submodules $LS^{\alpha,k}(\PC)$ effectively, it will be useful to fit $LS^{\alpha,k}(\PC)$ into a chain complex whose pieces are easier to understand.  In the next section we describe such a complex.

\section{An Intersection Complex}\label{sec:Int}
In this section we introduce a Cech-type complex for finite sums of submodules of a given $S$-module $M$ and give a criterion for its exactness.  We apply this to the submodules $LS^{\alpha,k}(\PC)\subset C^\alpha(\PC)$ in \S~\ref{sec:splineint}.

For an integer $N$, let $I(k)$ be the set of all subsets of size $k\ge 1$ formed from the index set $\{1,\ldots,N\}$.  Thinking of $I\in I(k)$ as a $k$-simpex of the $N$-simplex $\Delta$, we have the complex $\Delta_\bullet(S)$ with $\Delta_k(S)=\oplus_{I\in I(k)} S$ below whose homology is the simplicial homology of $\Delta$ with coefficients in $S$.
\[
\begin{array}{lr}
\Delta_\bullet(S): & \displaystyle S \xrightarrow{\delta_{N-1}} \bigoplus_{I\in I(N-1)} S \xrightarrow{\delta_{N-2}} \ldots \xrightarrow{\delta_k} \bigoplus_{I\in I(k)} S \xrightarrow{\delta_{k-1}} \ldots \xrightarrow{\delta_1} \bigoplus_{i=1}^N S \rightarrow 0
\end{array}
\]
If $k>0$ and $e_I\in \bigoplus_{I\in I(k)} S$ is the idempotent corresponding to $I=\{i_1,\ldots,i_k\}\in I(k)$, then 
\[
\delta_k(e_I)=\sum_{j=1}^k (-1)^{j-1} e_{I\backslash i_j}
\]
It will be convenient to augment this complex with a final map $\bigoplus_{i=1}^N S \xrightarrow{\epsilon} S$ defined by $\delta_0(e_i)=1$ for every $i$.  We denote this augmented complex as $\Delta^a_\bullet(S)$.  The homology of $\Delta^a_\bullet(S)$ computes the \textit{reduced} homology of $\Delta$ with coefficients in $S$.  We extend these complexes to an $S$-module $M$ by tensoring; let $\Delta_\bullet(M)$ denote $\Delta_\bullet(S)\otimes_S M$ and $\Delta^a_\bullet(M)$ denote $\Delta^a_\bullet(S)\otimes_S M$.

Now suppose $\M=\{M_1,\ldots,M_N\}$, where each $M_i$ is a submodule of $M$.  For $I\subset\{1,\ldots,N\}$ let $M_I$ denote the intersection $\cap_{i\in I} M_i$.  Define submodules $C_k(\M)=\oplus_{I\in I(k)} M_I\subset \oplus_{I\in I(k)} M = \Delta_k(M)$.  Since $M_I \subset M_{I\backslash i}$ for every $i\in I$, the differential $\delta_k$ of $\Delta_\bullet(M)$ restricts to a map $\delta_k:C_k(\M) \rightarrow C_{k-1}(\M)$, so $C_\bullet(\M)$ is a subcomplex of $\Delta_\bullet(M)$.  For example, if $N=2$, $C_\bullet(\M)$ is the complex
\[
0\rightarrow M_{12} \xrightarrow{\delta_1} M_1 \bigoplus M_2,
\]
where $\delta_1(m)=(-m,m)$.  Given any submodule $M'\subset M$ containing all the $M_i$, we may augment $C_\bullet(\M)$ with the map
\[
\bigoplus\limits_{i=1}^N M_i \xrightarrow{\epsilon} M',
\]
where $\epsilon(m_1,\ldots,m_N)=m_1+\cdots+m_N$.  We denote this augmented complex by $C^a_\bullet(\M,M')$.

Now consider the condition $(\star)$ on $\mathcal{M}$ given by
\[
\begin{array}{ll}
(\star) & M_I\cap (\sum_{i\in T} M_i)=\sum_{i \in T} (M_I\cap M_i)\\
\vspace{5 pt}
 & \mbox{ for every pair of subsets } I,T\subset \{1,\ldots,N\}
\end{array}
\]
We only need to check this condition on subsets $I,T$ with $I\cap T=\emptyset$, since if there is some $j\in I\cap T$ then $M_I\subset M_j$ and both sides are equal to $M_I$.

\begin{prop}\label{prop:augexact}
If $\mathcal{M}=\{M_1,\ldots,M_N\}$ satisfies $(\star)$ then $H_i(C^a_\bullet(\M,M))=0$ for $i>0$ and $H_0(C^a_\bullet(\M,M))=M/(\sum_{i=1}^N M_i)$.
\end{prop}

\begin{proof}
The assertion $H_0(C^a_\bullet(\M,M))=M/(\sum_{i=1}^N M_i)$ is always true, so we prove $H_i(C^a_\bullet(\M,M))=0$ for $i>0$.  We proceed by induction on the cardinality $N$ of $\mathcal{M}$.  If $N=2$ then $C^a_\bullet(\M,M)$ is the complex
\[
0\rightarrow M_{12} \xrightarrow{\delta_1} M_1\oplus M_2 \rightarrow M \rightarrow 0
\]
which satisfies the conclusion of Proposition~\ref{prop:augexact}.  Now suppose $N>2$.  Let $\M'=\{M_1,\ldots,M_{N-1}\}$ and $\N=\{M_{1,N},\ldots, M_{N-1,N}\}$, where $M_{i,j}=M_i\cap M_j$.  We have a short exact sequence of complexes $0\rightarrow C^a_\bullet(\M',M) \rightarrow C^a_\bullet(\M,M) \rightarrow C^a_\bullet(\N,M_N)(-1) \rightarrow 0$, shown below.  Here $C(i)$ denotes the complex $C$ with shifted grading $C(i)_j=C_{i+j}$.  This short exact sequence follows from the fact that $C^a_\bullet(\M,M)$ can be constructed as the mapping cone of the (appropriately signed) inclusion $C^a_\bullet(\N,M_N)\hookrightarrow C^a_\bullet(\M',M)$.  It is also not difficult to check exactness of this sequence directly.

\begin{flushleft}
\begin{tikzcd}
& & 0 \ar{d} &0\ar{d} & 0 \ar{d} \\
C^a_\bullet(\M',M) & 0 \ar{r}\ar{d} & M_{1,\ldots,N-1} \ar{d} \ldots\ar{r}{\delta'_{N-1}} & \bigoplus\limits_{i=1}^{N-1} M_i \ar{r}{\delta'_1}\ar{d} & M \ar{d} \\
C^a_\bullet(\M,M) & M_{1,\ldots,N} \ar{r}{\delta_N}\ar{d} & \bigoplus\limits_{I\in I(N-1)} M_I \ar{d} \ldots \ar{r}{\delta_2} & \bigoplus\limits_{i=1}^N M_i \ar{r}{\delta_1}\ar{d} & M \ar{d} \\
C^a_\bullet(\N,M_N)(-1) & M_{1,\ldots,N} \ar{r}{\delta''_N}\ar{d} & \bigoplus\limits_{\substack{I\in I(N-1)\\ N\in I}} M_I \ar{d} \ldots \ar{r}{\delta''_2} & M_N \ar{r}{\delta''_1}\ar{d} & 0 \\
& 0 &0 & 0 &
\end{tikzcd}
\end{flushleft}

Clearly $\M'$ satisfies $(\star)$, inheriting the necessary conditions from the fact that $\M$ satisfies $(\star)$.  Since $|\M'|=N-1$, $H_i(C^a_\bullet(\M',M))=0$ for $i>0$ by induction.  We claim $\mathcal{N}$ also satisfies $(\star)$.  Interpreted for the set $\mathcal{N}$, the condition $(\star)$ is
\[
\begin{array}{ll}
(\star\star) & M_{I\cup {N}} \cap (\sum_{i\in T} M_{i,N})=\sum_{i \in T} M_{I\cup{N}}\cap M_i\\
\vspace{5 pt}
 & \mbox{ for every pair of subsets }I, T\subset \{1,\ldots,N-1\}
\end{array}
\]
First note that for any subset $T\subset\{1,\ldots,N-1\}$, $\sum_{i\in T} M_{i,N}=M_N\cap (\sum_{i\in T} M_i)$ since $\M$ satisfies $(\star)$.  So the left hand side of $(\star\star)$ is equivalent to $M_{I\cup {N}} \cap (\sum_{i\in T} M_i)$.  Again, since $\M$ satisfies $(\star)$, $M_{I\cup {N}} \cap (\sum_{i\in T} M_i)=\sum_{i\in T} M_{I\cup{N}} \cap M_i$.  So $\N$ satisfies $(\star)$.  Since $|\N|=N-1$, $H_i(C^a_\bullet(\N,M_N)(-1))=0$ for $i>1$ by induction.  It follows from the long exact sequence in homology that
\[
H_i(C^a_\bullet(\M,M))=0
\]
for $i>1$.  For $i=1$ we have the exact sequence

\begin{center}
\begin{tikzcd}
0 \rar
& H_1(C^a_\bullet(\M,M)) \rar & H_1(C^a_\bullet(\N,M_N)(-1)) 
\ar[out=0,in=180,looseness=2,overlay]{dl} & \\
& H_0(C^a_\bullet(\M',M)) \rar & H_0(C^a_\bullet(\M,M)) \rar & 0
\end{tikzcd}
\end{center}

We have
\[
\begin{array}{rl}
H_1(C^a_\bullet(\N,M_N)(-1))= & \dfrac{M_N}{\sum_{i=1}^{N-1} M_i\cap M_N}\\[15 pt]
=&\dfrac{M_N}{M_N\cap(\sum_{i=1}^{N-1} M_i)}\\[15 pt]
=&\dfrac{M_N + \sum_{i=1}^{N-1}M_i}{\sum_{i=1}^{N-1} M_i},
\end{array}
\]
where the second equality comes from the fact that $\M$ satisfies $(\star)$.  But this is precisely the kernel of the natural surjection
\[
H_0(C^a_\bullet(\M',M))=\dfrac{M}{\sum_{i=1}^{N-1}M_i} \rightarrow \dfrac{M}{\sum_{i=1}^N M_i}=H_0(C^a_\bullet(\M,M)).
\]
It follows that $H_1(C^a_\bullet(\M,M))=0$ and we are done.
\end{proof}

\begin{cor}\label{cor:exactintC}
Let $\M=\{M_1,\ldots,M_N\}$ be a set of submodules of $M$.  Then if $\M$ satisfies $(\star)$,
\[
C_\bullet(\M)\rightarrow \sum\limits_{i=1}^N M_i \rightarrow 0
\]
is exact.
\end{cor}

\section{Intersection complex for splines}\label{sec:splineint}

In this section we apply the complex constructed in \S~\ref{sec:Int} to the case of lattice-supported splines.  Recall from \S~\ref{sec:Lat} that $\Gamma^k_\PC$ is the poset of components appearing among lattice complexes of the form $\PC_W$ with the rank of $W$ at most $k$, ordered with respect to inclusion.  $\Gamma^{k,\textup{max}}_\PC$ is the set of maximal complexes appearing in $\Gamma^{k}_\PC$.  With this notation, Proposition~\ref{prop:top} states
\[
LS^{\alpha,k}(\PC):=\sum\limits_{\OC\in\Gamma^{k,\textup{max}}_\PC} C^\alpha_\OC(\PC),
\]
where $C^\alpha_\OC(\PC)\subset C^\alpha(\PC)$ is the subalgebra of splines vanishing outside of $\OC$.

Now set $\M_k=\{C^\alpha_{\QC}(\PC)|\QC\in\Gamma^{k,\textup{max}}_\PC\}$.  $LS^{\alpha,k}(\PC)$ fits into the complex
\[
C_\bullet(\M_k)\rightarrow LS^{\alpha,k}(\PC).
\]
If $\QC,\OC$ are pure $n$-dimensional subcomplexes of $\PC\subset\R^n$, then
\[
C^\alpha_{\QC}(\PC)\bigcap C^\alpha_{\OC}(\PC)=C^\alpha_{\QC\cap\OC}(\PC).
\]
This extends to any finite intersection, hence we may write
\[
C_i(\M_k)=\bigoplus\limits_{\QC} C^\alpha_{\QC}(\PC),
\]
where $\QC$ runs across all intersections of $i$ subcomplexes from $\Gamma^{k,\mbox{\emph{max}}}_\PC$.  As will be evident below, the same subcomplex can appear multiple times as an intersection.  We prove exactness of $C_\bullet(\M_1)$.
\begin{prop}\label{prop:fakeres}
Let $\M_1=\{C^\alpha_\tau(\PC)|\tau\in\PC^0_{n-1}\}$.  Then the augmented complex
\[
C_\bullet(\M_1)\rightarrow LS^{\alpha,1}(\PC)\rightarrow 0
\]
is exact.
\end{prop}
\begin{proof}
We show that $\M_1$ satisfies the condition $(\star)$ from the previous section; then by Corollary~\ref{cor:exactintC} the proposition will be proved.  First suppose given $m>1$ codimension one faces $\tau_1,\ldots,\tau_m$.  If these are all faces of a common facet $\sigma$, then $\mbox{st}(\tau_1)\cap\ldots\cap\mbox{st}(\tau_m)=\sigma$.  Otherwise, this intersection has dimension less than $n$ and no splines are defined on it.  Hence to show $(\star)$ for $\M_1$ amounts to showing that, given a set $T=\{\tau_1,\ldots,\tau_n\}$ of codimension one faces of $\PC$, the following equalities hold.  Keep in mind that for two subcomplexes $\OC,\QC$, $C^\alpha_{\OC}(\PC)\cap C^\alpha_{\QC}(\PC)=C^\alpha_{\OC\cap\QC}(\PC)$.
\begin{enumerate}
\item For any facet $\sigma\in\PC_n$,
\[
C^\alpha_\sigma(\PC)\bigcap\left(\sum\limits_{i=1}^n C^\alpha_{\mbox{st}(\tau_i)}(\PC)\right)=\sum\limits_{i=1}^n C^\alpha_{\sigma\cap\mbox{st}(\tau_i)}(\PC),
\]
\item For any codimension one face $\tau\in\PC^0_{n-1}$,
\[
C^\alpha_{\mbox{st}(\tau)}(\PC)\bigcap\left(\sum\limits_{i=1}^n C^\alpha_{\mbox{st}(\tau_i)}(\PC)\right)=\sum\limits_{i=1}^n C^\alpha_{\mbox{st}(\tau)\cap\mbox{st}(\tau_i)}(\PC)
\]
\end{enumerate}
(1) If $\sigma\subset\mbox{st}(\tau_i)$ for some $\tau_i\in T$, then both sides are equal to $C^\alpha_\sigma(\PC)$.  Otherwise both sides are trivial.
(2) If $\tau\in T$, then both sides are equal to $C^\alpha_\tau(\PC)$.  Otherwise, set $C^\alpha_T(\PC)=\sum_{i=1}^n C^\alpha_{\mbox{st}(\tau_i)}(\PC)$ and let $F\in C^\alpha_{\mbox{st}(\tau)}(\PC)\cap C^\alpha_T(\PC)$.  Since $\tau\notin T$, $F$ must vanish along $\tau$ to order $\alpha(\tau)$.  Letting $\sigma_1,\sigma_2$ be the two facets of $\mbox{st}(\tau)$, we see $F|_{\sigma_i}\in C^\alpha_{\sigma_i}(\PC)$ for $i=1,2$.  It follows that
\[
C^\alpha_{\tau}(\PC)\cap C^\alpha_T(\PC)=C^\alpha_{\sigma_1}(\PC)\cap C^\alpha_T(\PC)+C^\alpha_{\sigma_2}(\PC)\cap C^\alpha_T(\PC).
\]
Now by (1) the intersections $C^\alpha_{\sigma_i}(\PC)\cap C^\alpha_T(\PC)$ distribute.
\end{proof}

\begin{remark}
It would be interesting to know if Proposition~\ref{prop:fakeres} holds for any $\M_k$, where $k>1$.
\end{remark}

\begin{prop}\label{prop:M1Res}
Let $\M_1=\{C^\alpha_\tau(\PC)|\tau\in\PC^0_{n-1}\}$, where $\PC$ is a pure $n$-dimensional hereditary polytopal complex.  For a facet $\sigma\in\PC_n$, let $\partial^0(\sigma)$ denote the set of codimension one faces of $\sigma$ that are interior faces of $\PC$.  Set $\delta(\PC)=\max_{\sigma\in\PC_n}\{|\partial^0(\sigma)|\}$.  The complex $C_\bullet(\M_1)$ satisfies
\[
C_k(\M_1)=\left\lbrace
\begin{array}{ll}
\bigoplus\limits_{\tau\in\PC^0_{n-1}} C^\alpha_{\tau}(\PC) & \mbox{if }k=1\\
\bigoplus\limits_{|\partial^0(\sigma)|\ge k}\left( C^\alpha_{\sigma}(\PC)\right)^{\binom{|\partial^0(\sigma)|}{k}} & \mbox{if } 2\le k\le \delta(\PC)\\
0 & \mbox{if }k>\delta(\PC)
\end{array}
\right.
\]
\end{prop}
\begin{proof}
By definition $C_1(\M_1)$ is the direct sum of all the submodules of $\M_1$.  In general we have
\[
C_k(\M_1)=\bigoplus_{\QC} C^\alpha_{\QC}(\PC),
\]
where the direct sum runs over all subcomplexes $\QC\subset\PC$ which are intersections of $k$ distinct subcomplexes chosen from the set $\{\mbox{st}(\tau)|\tau\in\PC^0_{n-1}\}$.  If $k\ge 2$ then $\QC$ is the intersection of two or more stars of codimension one faces, say $\mbox{st}(\tau_1),\mbox{st}(\tau_2),\ldots,\mbox{st}(\tau_k)$.  Hence $\QC$ contains at most one facet, and that facet must have $\tau_1,\ldots,\tau_k$ as faces.  So if $k\ge 2$,
\[
C_k(\M_1)=\bigoplus\limits_{|\partial^0(\sigma)|\ge k}\left( C^\alpha_{\sigma}(\PC)\right)^{\binom{|\partial^0(\sigma)|}{k}},
\]
where $|\partial^0(\sigma)|$ is the number of edges of $\sigma$ which are interior to $\PC$.  From this we also see that $C_k(\M_1)=0$ for $k>\delta(\PC)$.
\end{proof}

\section{Regularity}\label{sec:regularity}
In this section we briefly summarize some commutative algebra.  The first chapter of~\cite{Syz} is an excellent introduction to the graded approach we take here.  Most of the material of this section comes from this source.

Let $M$ be a graded module over the polynomial ring $S=k[x_0,\ldots,x_n]$, where $k$ is a field of characteristic $0$.  Let $HF(M,d)=\dim_k M_d$ denote the Hilbert function of $M$ in degree $d$.  A standard result states that for $d\gg 0$, $HF(M,d)$ agrees with a polynomial function $HP(M,d)$, the Hilbert polynomial of $M$.  As we noted in the introduction, the largest integer $d$ for which $HF(M,d)\neq HP(M,d)$ is called the \textit{postulation number} of $M$, denoted by $\wp(M)$.  The degree of $HP(M,d)$ is one less than the Krull dimension of $M$, denoted $\mbox{dim}(M)$.  The \textit{codimension} of $M$ is defined by $\mbox{dim}(S)-\mbox{dim}(M)=n+1-\mbox{dim}(M)$.  $M$ has a minimal graded free resolution
\[
F_\bullet: \mbox{ }0\rightarrow F_\delta \xrightarrow{\phi_r} F_{r-1}\xrightarrow{\phi_{r-1}} \cdots \xrightarrow{\phi_1} F_0,
\]
with $\mbox{coker }\phi_1=M$.  The index $\delta$ of the final free module appearing in this resolution is called the projective dimension of $M$, denoted $\mbox{pd}(M)$.  For an integer $a$, let $S(a)$ denote the polynomial ring with grading shifted by $a$, so $S(a)_d\cong S_{a+d}$.  An important invariant of the module $M$ which (among other things) governs when $HF(M,d)$ becomes polynomial is the \textit{Castelnuovo-Mumford regularity} of $M$.
\begin{defn}\label{defn:regdef}
Let $M$ be a graded $S$ module and $F_\bullet \rightarrow M$ the minimal free resolution of $M$, with $F_i\cong \bigoplus\limits_{j} S(-a_{ij})$.  The Castelnuovo-Mumford regularity of $M$, denoted $\reg(M)$, is defined by
\[
\reg(M)=\max_{i,j}\{a_{i,j}-i\}.
\]
\end{defn}

\begin{remark}\label{rm:degrees}
Note that, according to this definition, $\reg(M)$ bounds the minimal degree of generators of $M$ as an $S$-module.
\end{remark}

From Definition~\ref{defn:regdef} one derives the following theorem.  Recall an $S$-module $M$ is \textit{Cohen-Macaulay} if $\mbox{codim}(M)=\mbox{pd}(M)$.
\begin{thm}~\cite[Theorem~4.2]{Syz}\label{thm:HFHP}
Let $M$ be a finitely generated graded module over $S$.  Then 
\begin{enumerate}
\item $HF(M,d)=HP(M,d)$ for $d\ge \reg(M)+\mbox{pd}(M)-n$.  Equivalently, $\wp(M)\le \reg(M)+\mbox{pd}(M)-n-1$.
\item If $M$ is a Cohen-Macaulay module, the bound in (1) is sharp.
\end{enumerate}
\end{thm}
Another characterization of regularity is obtained via local cohomology, so we introduce this notion.  See ~\cite[Appendix~1]{Syz} for more details.  Let $Q$ be an ideal of $S$.  The local cohomology modules $H^i_Q(M)$ of $M$ with respect to $Q$ are the right derived functors of the the $Q$-torsion functor $H^0_Q(\underline{\mbox{\hspace{6 pt}}})$, where
\[
H^0_Q(M)=\{x\in M| Q^jx=0\mbox{ for some } j\ge 0\}.
\]
We will only be concerned with the case $Q=m$, where $m=(x_0,\ldots,x_n)$ is the graded maximal ideal of $S$.

\begin{thm}[Theorem~4.3 of~\cite{Syz}]\label{thm:CohoReg}
Let $m\subset S$ be the maximal ideal of $S$ and $M$ a graded $S$-module.  Then
\[
\reg(M)=\max\limits_i(\max\limits_e \{e|H^i_m(M)_e\neq 0\}+i)
\]
\end{thm}
The benefit of this description of regularity is that it interacts well with short exact sequences.  For instance, the following result is a straightforward application of Theorem~\ref{thm:CohoReg}.
\begin{prop}\cite[Corollary~20.19]{Eis}\label{prop:regseq}
Let $0\rightarrow A \rightarrow B \rightarrow C \rightarrow 0$ be a graded exact sequence of finitely generated $S$ modules.  Then
\begin{enumerate}
\item $\reg(A) \le \max \{\reg(B),\reg(C)+1\}$
\item $\reg(B) \le \max \{\reg(A),\reg(C)\}$
\item $\reg(C) \le \max \{\reg(A)-1,\reg(B)\}$
\end{enumerate}
\end{prop}
Proposition~\ref{prop:regseq} can be extended to bound the regularity of a module appearing in an exact sequence of any length by breaking the exact sequence into short exact pieces.  We will use the following corollary to Proposition~\ref{prop:regseq}.
\begin{cor}\label{cor:lesreg}
Let $m\ge 0$ and
\[
0\rightarrow C_m \rightarrow C_{m-1} \rightarrow \ldots \rightarrow C_0 \rightarrow M \rightarrow 0
\]
an exact sequence of $S$-modules.  Then
\[
\reg(M)\le \max_i\{\reg (C_i)-i\}
\]
\end{cor}

One more concept that is relevant to our situation is that of depth.  The \textit{depth} of a graded $S$-module $M$ with respect to the homogeneous maximal ideal $m$, denoted $\mbox{depth}(M)$, is the length of a maximal sequence $\{f_1,\ldots,f_k\}\subset m$ satisfying that $f_1$ is a non-zerodivisor on $M$ and $f_l$ is a non-zerodivisor on $M/(\sum_{i=1}^{l-1} f_iM)$ for $l=2,\ldots,k$.  Such a sequence is called an $M$-sequence.  We will use the following result of Auslander and Buchsbaum to move back and forth between the notions of depth and projective dimension.
\begin{thm}[Auslander-Buchsbaum]\label{thm:AB}
Let $M$ be an $S=k[x_0,\ldots,x_n]$-module.  Then
\[
\mbox{depth}(M)+\mbox{pd}(M)=n+1.
\]
\end{thm}
Observe that, according to this formula, $\mbox{pd}(M)\le n+1$.  This inequality is known as the Hilbert syzygy theorem.

The following proposition is one of the ingredients used in the proof of the Gruson-Lazarsfeld-Peskine theorem on bounding the regularity of curves in projective space~\cite[Proposition~5.5]{Syz}.  It is the main tool we will use for bounding regularity of spline modules.

\begin{prop}\label{prop:RegDepth}
Let $M$ be an $S$-module and $N\subset M$ a submodule of $M$ with $\mbox{dim}(M/N)<\mbox{depth}(M)$, or equivalently $\mbox{codim}(M/N)>\mbox{pd}(M)$.  Then $\reg(M) \le \reg(N)$.
\end{prop}
\begin{proof}
We prove $\reg(M) \le \reg (N)$ if $\mbox{dim}(M/N)<\mbox{depth}(M)$.  The equivalence of the statements $\mbox{dim}(M/N)<\mbox{depth}(M)$ and $\mbox{codim}(M/N)>\mbox{pd}(M)$ follows directly from Theorem~\ref{thm:AB}.  
Set $d=\mbox{depth}(M)$.  By~\cite[Proposition~A1.16]{Syz}, $H^i_m(M)=0$ for $i<d$ and $H^i_m(M/N)=0$ for $i>\mbox{dim}(M/N)$.  The long exact sequence in local cohomology resulting from the short exact sequence
\[
0\rightarrow N \rightarrow M \rightarrow M/N \rightarrow 0
\]
yields a surjection $H^d_m(N) \twoheadrightarrow H^d_m(M)$ and isomorphisms $H^i_m(N)\cong H^i_m(M)$ for $i>d$.  Since $H^i_m(M)=0$ for $i<d$, Theorem~\ref{thm:CohoReg} yields $\reg(N)\ge \reg(M)$.
\end{proof}

\subsection{High degree generators for splines}

We give a construction motivating the regularity bounds we derive in Corollary~\ref{cor:freegens}, Theorem~\ref{thm:main2}, and Theorem~\ref{thm:main3}.  These results suggest that in general regularity bounds for $C^\alpha(\PC)$ might be obtained by taking the maximal sum of smoothness parameters $\alpha(\tau)+1$ appearing in certain subcomplexes of $\PC$.  In the following example, starting with a polytope $\sigma\subset\R^n$, we construct a polytopal complex $\PC$ so that $\sigma\in\PC_n$ and $C^\alpha(\wPC)$ ($C^\alpha(\PC)$ if $\PC$ is central) has a minimal generator supported the facet $\wsigma$ ($\sigma$ if $\PC$ is central).  Such generators have degree $\sum_{\tau\in\sigma_{n-1}} \alpha(\tau)+1$.  Since $\reg(C^\alpha(\PC))$ in particular bounds the degrees of generators of $C^\alpha(\PC)$ (see Remark~\ref{rm:degrees}), this construction indicates that a bound on $\reg(C^\alpha(\PC))$ will need to be at least as large as the maximal sum of smoothness parameters over codimension one faces occurring in any facet of $\PC$ (or at least boundary facets - see Conjecture~\ref{conj:c1}).  This example generalizes the construction in~\cite[Theorem~5.7]{LS}.

For simplicity we restrict the construction to the case of uniform smoothness without imposing boundary vanishing.  The generalization to arbitrary smoothness parameters should be clear.

\begin{exm}\label{ex:HighGen}
Suppose that $A\subset\R^n$ is a polytope with a codimension one face $\tau\in A_{n-1}$ so that $\partial A\setminus\tau$ is the graph of a piecewise linear function over $\tau$.  Remark~\ref{rm:polygraph} below shows that this can be accomplished for any polytope by a projective change of coordinates.

For instance this is true if $A$ is the join of $\tau$ with the origin $\mathbf{0}\in\R^n$.  Let $l_\tau$ be a choice of affine form vanishing on $\tau$ and let $x_1,\ldots,x_n$ be coordinates on $\R^n$.  We further assume that
\begin{enumerate}
\item $\tau$ is parallel to the coordinate hyperplane $x_n=0$
\item $A$ lies between the hyperplanes $x_n=0$ and $l_\tau=0$.
\item For any two codimension one faces $\gamma_1,\gamma_2\in A_{n-1}\setminus\tau$ , $\mbox{aff}{\gamma_1}$ and $\mbox{aff}(\gamma_2)$ intersect the coordinate hyperplane $x_n=0$ in distinct linear subspaces of codimension $2$.
\end{enumerate}
(1) can be obtained by rotating the original polytope, (2) and (3) can be obtained by translation.  If $A$ is the join of $\tau$ with the origin, (3) may be obtained by slight perturbations of the non-zero vertices of $A$ (within the plane $l_\tau=0$).

Let $B$ be the reflection of $A$ across the hyperplane $x_n=0$.  For a face $\gamma\in A$, let $\bar{\gamma}$ denote the corresponding face of $B$ obtained by reflection.  For $\gamma\in A_{n-1}\setminus\tau$, let $\sigma(\gamma)$ denote the polytope formed by taking the convex hull of $\gamma$ and $\bar{\gamma}$.  Now define $\PC(A)$ as the polytopal complex with facets $A,B$ and $\{\sigma(\gamma)|\gamma\neq\tau\in A_{n-1}\}$.  See Figure~\ref{fig:highgens} for examples of this construction in $\R^2$ and $\R^3$.
\begin{figure}[htp]
\centering
\begin{subfigure}[b]{.45\linewidth}
\includegraphics[scale=.5]{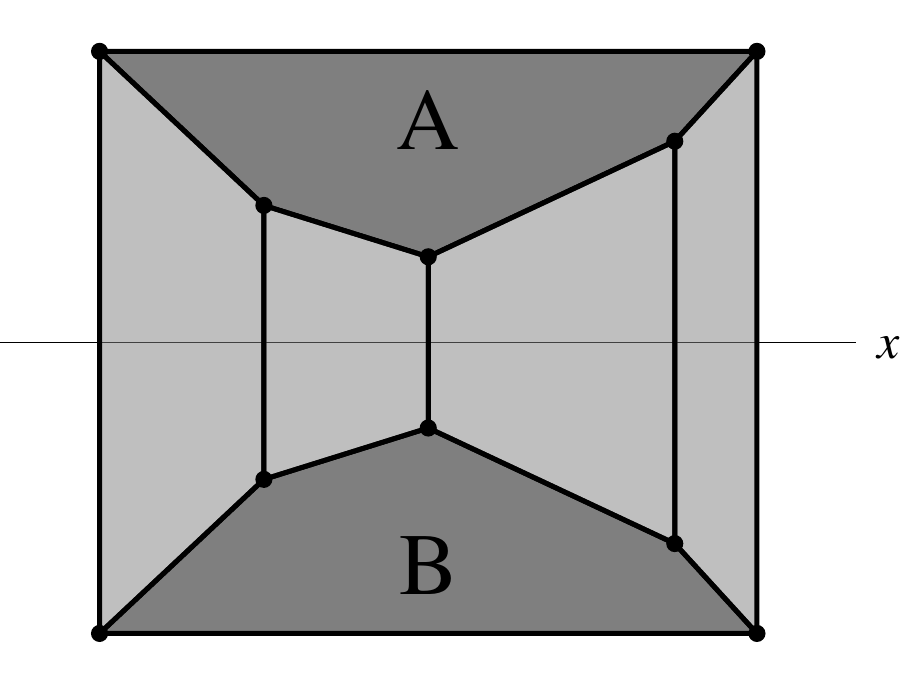}
\caption{$\PC(A)$ for $A\subset\R^2$}\label{fig:HGNC}
\end{subfigure}
\begin{subfigure}[b]{.45\linewidth}
\includegraphics[scale=.5]{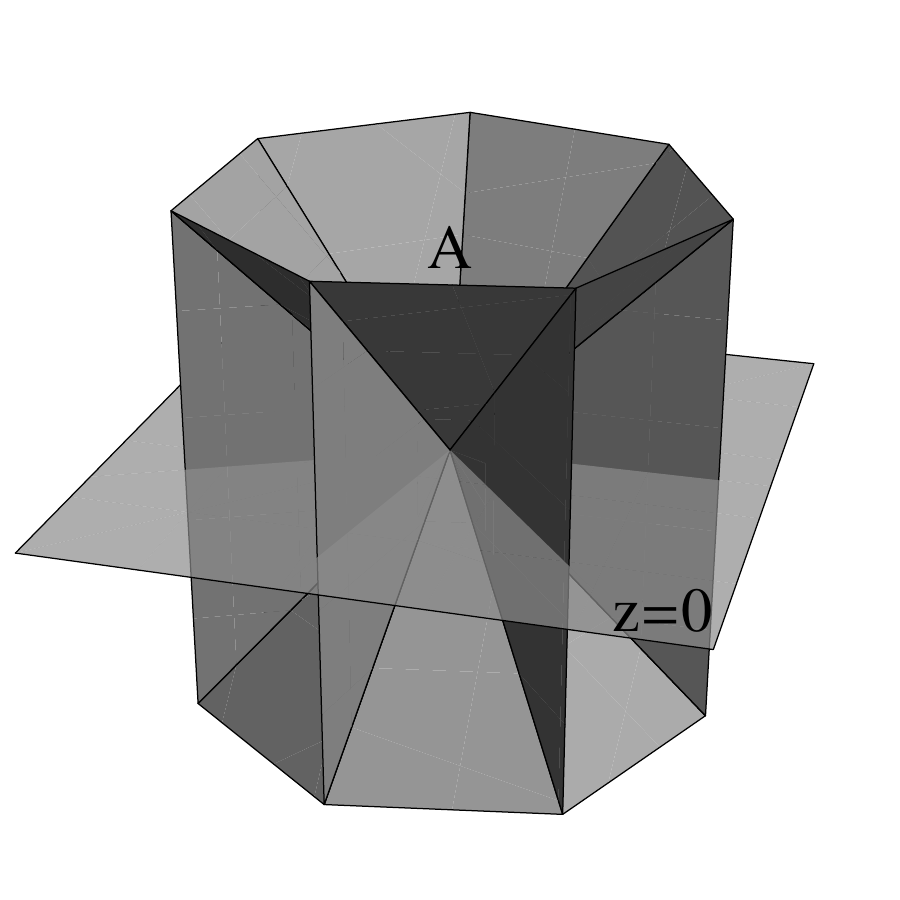}
\caption{$\PC(A)$ for central $A\subset\R^3$}\label{fig:HGC}
\end{subfigure}
\caption{}\label{fig:highgens}
\end{figure}
Take the cone $\widehat{\PC(A)}\subset\R^n$ over $\PC(A)$ and consider the graded $S=\R[x_0,\ldots,x_n]$-module $C^r(\widehat{\PC(A)})$.  Let $\phi_B:C^r(\widehat{\PC(A)})\rightarrow S$ be the $S$-linear map obtained by restricting splines $F\in C^r(\widehat{\PC(A)})$ to the facet $\widehat{B}$.  This is a splitting of the inclusion $S\rightarrow C^r(\widehat{\PC(A)})$ as global polynomials on $\widehat{\PC(A)}$.  Let $NT^r(\widehat{\PC(A)})$ be the kernel of $\phi_B$.  Then
\[
C^r(\widehat{\PC(A)})\cong S \oplus NT^r(\widehat{\PC(A)}).
\]
Let $S'=\R[x_0,\ldots,x_{n-1}]$ and, for $f\in S$, set $\overline{f}=f(x_0,\ldots,x_{n-1},0)$.  Define an $S$-linear map $\phi:C^r(\widehat{\PC(A)})\cong S\oplus NT^r(\widehat{\PC(A)}) \rightarrow S'$ by
\[
(f,F)\rightarrow \overline{F_{\widehat{A}}},
\]
where $f\in S$, $F\in NT^r(\widehat{\PC(A)})$, and $F_{\widehat{A}}$ is the restriction of $F$ to the facet $\widehat{A}$.
Set $\Lambda(A)=\prod\limits_{\gamma\neq\tau\in A_{n-1}} L^{r+1}_\gamma$, where $L_\gamma=l_{\wgamma}$ is a choice of homogeneous form vanishing on $\wgamma$.  We claim that the image of $\phi$ is the principal ideal
\[
I=\langle \overline{\Lambda(A)}\rangle.
\]
$\phi$ is surjective since the spline $G(A)$, defined by
\[
G(A)_\sigma=\left\lbrace
\begin{array}{ll}
0 & \sigma\neq A\\
\Lambda(A) & \sigma=A,
\end{array}
\right.
\]
goes to the generator of $I$ under $\phi$.  To see that $\mbox{im}(\phi)\subset I$, let $F\in NT^r(\widehat{\PC(A)})$.  Then, since $F_{\widehat{B}}=0$, $L^{r+1}_{\bar{\gamma}}|F_{\sigma(\gamma)}$ for every $\bar{\gamma}\neq\bar{\tau}\in B_{n-1}$.  We also have $L^{r+1}_\gamma|(F_{\widehat{A}}-F_{\widehat{\sigma(\gamma)}})$ for every $\gamma\neq\tau\in A_{n-1}$.  Hence $F_{\widehat{A}}\in \cap_{\gamma\neq\tau\in A_{n-1}} \langle L^{r+1}_\gamma,L^{r+1}_{\bar{\gamma}}\rangle$.  But $L_\gamma$ and $L_{\bar{\gamma}}$ differ at most by a scalar multiple and a sign on the variable $x_n$, so $\overline{L_\gamma}=\overline{L_{\bar{\gamma}}}$ and
\[
\phi(F)\in\bigcap\limits_{\gamma\neq\tau\in A_{n-1}} \langle \overline{L^{r+1}_\gamma} \rangle=\langle \prod\limits_{\gamma\neq\tau\in A_{n-1}} \overline{L^{r+1}_\gamma} \rangle=\langle \overline{\Lambda(A)}\rangle
\]
as claimed.  Property (3) above is used in the first equality - this guarantees all the forms $\overline{L_\gamma}$ are distinct.  It follows that the spline $G(A)$, which is supported only on the facet $\widehat{A}$ and generates splines supported on $\widehat{A}$, is a minimal generator of $C^r(\widehat{\PC(A)})$.

If $A$ is the join of $\tau$ with $\mathbf{0}$, then $\PC(A)$ is central and $C^r(\PC(A))$ is graded over the polynomial ring $R=\R[x_1,\ldots,x_n]$.  In this case it is unnecessary to take the cone over $\PC(A)$ above.
\end{exm}

\begin{remark}\label{rm:polygraph}
Given a convex polytope $A\subset\R^n\subset\mathbb{P}^n_\R$ and a choice $\tau$ of codimension one face, there is a projective change of coordinates which makes $\partial A \setminus \tau$ into the graph of a piecewise linear function over $\tau$.  If $A$ is the join of $\tau$ with the origin $\mathbf{0}\in\R^n$, then this is easily done by a linear transformation.  Otherwise, this can be accomplished by choosing a hyperplane $H\subset\R^n$ which is parallel to $\tau$ and very close to $P$ without intersecting $P$.  Then make a projective change of coordinates which sends $H$ to the hyperplane at infinity (this argument is due to Sergei Ivanov).  As long as $H$ is chosen close enough to $\tau$, this has the effect of making the face $\tau$ huge and the rest of the polytope the graph of a piecewise linear function over $\tau$ (once we restrict to affine coordinates again).  Hence, given any polytope $A\subset\R^n$ and a choice of codimension one face $\tau\in A_{n-1}$, the construction in Example~\ref{ex:HighGen} allows us to build a polytopal complex $\PC(A)$ so that $\partial^0 A=\partial A\setminus\tau$ and the generator of $C^r_{\widehat{A}}(\widehat{\PC(A)})$ is a minimal generator of $C^r(\widehat{\PC(A)})$.
\end{remark}

\begin{remark}
The construction in Example~\ref{ex:HighGen} is inherently nonsimplicial.  Some other construction needs to be used to obtain high degree generators in the simplicial case.  In the planar simplicial case, there is an example in~\cite{2r} of a planar simplicial complex $\Delta$ with minimal generator in degree $2r+2$.
\end{remark}

\section{Bounding Regularity for Low Projective Dimension}\label{sec:LowPD}
In this section we combine the observations so far to bound the regularity of the spline algebra $C^\alpha(\PC)$, where $\PC\subset\R^{n+1}$ is a central, pure, hereditary, $(n+1)$-dimensional polytopal complex.  Recall a central complex is one in which the intersection of all interior codimension one faces is nonempty.  We assume this intersection contains the origin and that $\alpha(\tau)=-1$ for every codimension one face $\tau\in\PC_n$ so that $\mathbf{0}\notin\mbox{aff}(\tau)$; this makes the ring $C^\alpha(\PC)$ a graded $S=\R[x_0,\ldots,x_n]$-algebra with respect to the standard grading on $S$.  The following corollary is critical to our analysis.

\begin{cor}\cite[Proposition~3.4]{Modules}\label{cor:SplinePD}
If $\PC$ is a central, pure, hereditary, $(n+1)$-dimensional polytopal complex, then 
\begin{enumerate}
\item $\mbox{pd}(C^\alpha(\PC))\le n-1$
\item $\wp(C^\alpha(\PC))\le \reg(C^\alpha(\PC))-2$.
\end{enumerate}
\end{cor}
\begin{proof}
(1) follows from Lemma~\ref{lem:seq1}.  $C^\alpha(\PC)$ is the kernel of a map between free $S$-modules, so it is a second syzygy module.  By the Hilbert syzygy theorem, any $S$-module has projective dimension at most $n+1$.  Since $C^\alpha(\PC)$ is a second syzygy module, $\mbox{pd}(C^\alpha(\PC))\le n-1$.  (2) follows from (1) and Theorem~\ref{thm:HFHP}.
\end{proof}

\begin{thm}\label{thm:main0}
Let $\PC\subset\R^{n+1}$ be a pure $(n+1)$-dimensional hereditary polytopal complex which is central.  Then
\[
\reg (C^\alpha(\PC)) \le \reg (LS^{\alpha,n-1}(\PC))
\]
More generally, if $\mbox{pd}(C^\alpha(\PC))\le k$, then
\[
\reg (C^\alpha(\PC)) \le \reg (LS^{\alpha,k}(\PC))
\]
\end{thm}
\begin{proof}
The first statement follows from the second by Corollary~\ref{cor:SplinePD}.  To prove the second statement, note that by Theorem~\ref{thm:latticegens}, the cokernel of the inclusion $LS^{\alpha,k}(\PC)$ has codimension at least $k+1$.  By Proposition~\ref{prop:RegDepth}, $\reg (C^\alpha(\PC)) \le \reg (LS^{\alpha,k}(\PC))$.
\end{proof}

To simplify the statements of later results, we introduce some additional notation.  Given a pure $(n+1)$-dimensional subcomplex $\QC\subset\PC$, let $\partial(\QC)$ denote the set of $n$ dimensional boundary faces of $\QC$.  Define
\[
\Lambda(\QC)=\prod\limits_{\substack{\gamma\in (\partial(\QC))_n}} l^{\alpha(\gamma)+1}_{\gamma}
\]
and set
\[
\lambda(\QC)=\mbox{deg}(\Lambda(\QC))=\sum_{\gamma\in(\partial(\QC))_n}(\alpha(\gamma)+1).
\]
As a first application of Theorem~\ref{thm:main0}, we give a bound on the degree of generators of $C^\alpha(\PC)$ when $C^\alpha(\PC)$ is free.

\begin{cor}\label{cor:freegens}
Suppose $C^\alpha(\PC)$ is free and set $f(\PC)=\max\{\lambda(\sigma)|\sigma\in\PC_{n+1}\}$.  Then $C^\alpha(\PC)$ is generated in degrees at most $f(\PC)$.
\end{cor}
\begin{proof}
For a free module, regularity is the maximum degree of generators (this follows from Definition~\ref{defn:regdef}), so we need to show $\reg (C^\alpha(\PC))\le f(\PC)$.
$C^\alpha(\PC)$ is free iff $\mbox{pd}(C^\alpha(\PC))=0$.  By Theorem~\ref{thm:main0},
\[
\reg (C^\alpha(\PC)) \le \reg (LS^{\alpha,0}(\PC)).
\]
By Corollary~\ref{cor:LowkDescription}, $LS^{\alpha,0}=\sum_{\sigma\in\PC_{n+1}} C^\alpha_{\sigma}(\PC)$.  Since the support of each summand is disjoint, this is a direct sum, so $\reg (LS^{\alpha,0}(\PC))=\max\{\reg ( C^\alpha_{\sigma}(\PC))|\sigma\in\PC_{n+1}\}$.  Also, $C^\alpha_{\sigma}(\PC)$ consists of splines $F$ supported on the single facet $\sigma$.  Such splines are characterized by $F|_{\sigma}$ being a polynomial multiple of $\Lambda(\sigma)$.  It follows that $C^\alpha_{\sigma}(\PC)\cong S(-\lambda(\sigma))$.  Hence
\[
\reg (LS^{\alpha,0}(\PC))=\max\{\lambda(\sigma)|\sigma\in\PC_{n+1}\}=f(\PC).
\]
\end{proof}

We now apply Theorem~\ref{thm:main0} to the case where $C^\alpha(\PC)$ has projective dimension at most one.  In particular, this includes central complexes in $\R^3$ by Corollary~\ref{cor:SplinePD}.

\begin{thm}\label{thm:main1}
Suppose $\mbox{pd}(C^\alpha(\PC))\le 1$.  Let $f(\PC)=\max\{\lambda(\sigma)|\sigma\in\PC_{n+1}\}$ and $T=\max\limits_{\tau\in\PC^0_n}\{\reg (C^\alpha_{\tau}(\PC))\}$.  Then $\reg (C^\alpha(\PC))\le \max\{f(\PC)-1,T\}$.
\end{thm}
\begin{proof}
By Corollary~\ref{thm:main0},
\[
\reg (C^\alpha(\PC)) \le \reg (LS^{\alpha,1}(\PC)).
\]
By Proposition~\ref{prop:fakeres}, $LS^{\alpha,1}(\PC)$ fits into the exact sequence
\[
C_\bullet(\M_1)\rightarrow LS^{\alpha,1}(\PC) \rightarrow 0.
\]
From Proposition~\ref{prop:M1Res},
\[
C_k(\M_1)=\left\lbrace
\begin{array}{ll}
\bigoplus\limits_{\tau\in\PC^0_n} C^\alpha_{\tau}(\PC) & \mbox{if }k=1\\
\bigoplus\limits_{|\partial^0(\sigma)|\ge k}\left( C^\alpha_{\sigma}(\PC)\right)^{\binom{|\partial^0(\sigma)|}{k}} & \mbox{if } 2\le k\le \delta(\PC)\\
0 & \mbox{if } k>\delta(\PC)
\end{array}
\right.,
\]
where $\delta(\PC)=\max_{\sigma\in\PC_{n+1}}\{|\partial^0(\sigma)|\}$.  As we saw in the proof of Corollary~\ref{cor:freegens}, $C^\alpha_{\sigma}(\PC)\cong S(-\lambda(\sigma))$, hence
\[
\reg (C_k(\M_1))=\max\{\lambda(\sigma)|\sigma\in\PC_{n+1}\}\le f(\PC)
\]
for every $k$ with $2\le k \le \delta(\PC)$.  Now the conclusion follows from Corollary~\ref{cor:lesreg}.
\end{proof}

At this point we see that to obtain more precise results for projective dimension one it is necessary to understand the ring $C^\alpha_\tau(\PC)$ of splines vanishing outside the star of a codimension one face.

\begin{prop}\label{prop:idsyz}
Let $\tau\in\PC^0_n$ be an interior codimension one face of $\PC$, and $\sigma_1,\sigma_2$ the two facets of $\mbox{st}(\tau)$, the star of $\tau$.  Set $L_\tau=l^{\alpha(\tau)+1}_\tau$, $L_1=\Lambda(\sigma_1)/L_\tau,L_2=\Lambda(\sigma_2)/L_\tau$.  Define the ideal $K(\tau)$ by
\[
K(\tau)=\left\langle L_1,L_2,L_\tau \right\rangle
\]
We have a graded isomorphism
\[
C^\alpha_\tau(\PC)\cong\left\lbrace
\begin{array}{ll}
S(-\mbox{deg }L_\tau-\mbox{deg }L_2)\oplus S(-\mbox{deg }L_1) & \mbox{if } L_1\in\left\langle L_2,L_\tau \right\rangle\\
S(-\mbox{deg }L_\tau-\mbox{deg }L_1)\oplus S(-\mbox{deg }L_2) & \mbox{if } L_2\in\left\langle L_1,L_\tau \right\rangle\\
S(-\mbox{deg }L_1-\mbox{deg }L_2)\oplus S(-\mbox{deg }L_\tau) & \mbox{if } L_\tau\in\left\langle L_1,L_2 \right\rangle\\
\syz( K(\tau)) & \mbox{otherwise},
\end{array}\right.
\]
where $\syz (K(\tau))$ is the module of syzygies on the ideal $K(\tau)$.
\end{prop}
\begin{proof}
Let $F\in C^\alpha_\tau(\PC)$ and set $F_1=F|_{\sigma_1},F_2=F|_{\sigma_2}$.  Then there are polynomials $G_1,G_2,G_3$ satisfying the following relations.
\[
\begin{array}{rl}
F_1= & G_1L_1\\
F_2= & G_2L_2\\
F_2-F_1= & G_3L_{\tau}
\end{array}
\]
Taking the alternating sum of the above equations yields
\begin{eqnarray}\label{eqn:syzrel}
G_1 L_1-G_2 L_2+G_3L_\tau= & 0.
\end{eqnarray}
Hence $F=(F_1,F_2)$ gives rise to a syzygy on the columns of the matrix
\[
M=
\begin{bmatrix}
L_1 & L_2 & L_\tau
\end{bmatrix}
\]
Now suppose given a syzygy $(G_1,G_2,G_3)$ on the columns of $M$.  We obtain a spline $F\in C^\alpha_\tau(\PC)$ by setting $F_1=G_1L_1,F_2=G_2L_2$, hence $C^\alpha_\tau(\PC)$ is isomorphic to the syzygies on the columns $M$.  If $K(\tau)$ is minimally generated by $L_1,L_2,$ and $L_\tau$, we obtain $C^\alpha_\tau(\PC)\cong \syz (K(\tau))$.  Otherwise we obtain the cases listed above.  For instance, if $L_1\in\langle L_2,L_\tau \rangle$, then there exist polynomials $f,g\in S$ so that $L_1=fL_2+gL_\tau$ and $\syz (M)$ is generated by
\[
\begin{bmatrix}
0\\
L_\tau\\
-L_2
\end{bmatrix},
\begin{bmatrix}
1\\
-f\\
-g
\end{bmatrix},
\]
of degrees $\mbox{deg }L_2+\mbox{deg }L_\tau $ and $\mbox{deg }L_1$, respectively.  The other cases follow similarly.
\end{proof}

\begin{prop}\label{prop:edgereg}
Let $\PC\subset\R^3$ be a central complex, and $\tau\in\PC^0_2$ a codimension one face of $\PC$.  Define
\[
\lambda(\tau)=\lambda(\mbox{st}(\tau))+\alpha(\tau)+1=\sum\limits_{\gamma\in(\mbox{st}(\tau))_2} \alpha(\gamma)+1.
\]
Then $\reg (C^\alpha_\tau (\PC))\le \lambda(\tau)-1$ unless $\alpha(\gamma)=-1$ for all $\gamma\neq\tau\in(\mbox{st}(\tau))$, when $\reg (C^\alpha_\tau(\PC))=\alpha(\tau)+1$.
\end{prop}
\begin{proof}
Let $L_1,L_2,L_\tau$ be as defined in proposition~\ref{prop:idsyz}.  Then
\[
\begin{array}{rl}
\mbox{deg }L_1= & \left(\sum\limits_{\gamma\in(\sigma_1)_2} (\alpha(\gamma)+1)\right)-\alpha(\tau)-1\\
\mbox{deg }L_1= & \left(\sum\limits_{\gamma\in(\sigma_2)_2} (\alpha(\gamma)+1)\right)-\alpha(\tau)-1\\
\mbox{deg }L_\tau= & \alpha(\tau)+1,
\end{array}
\]
If the ideal $K(\tau)=\langle L_1,L_2,L_\tau \rangle$ is not minimally generated by $L_1,L_2$, and $L_\tau$, then $C^\alpha_\tau(\PC)$ is free, generated in degrees indicated by Proposition~\ref{prop:idsyz}.  By that description $\reg (C^\alpha_\tau(\PC))\le \lambda(\tau)-1$ unless $\alpha(\gamma)=-1$ for all $\gamma\neq\tau\in(\mbox{st}(\tau))$, when $\reg (C^\alpha_\tau(\PC))=\alpha(\tau)+1$.  So assume $K(\tau)$ is minimally generated by $L_1,L_2,L_\tau$ and $C^\alpha_\tau(\PC)\cong\syz (K(\tau))$.

We define a submodule $N(\tau)$ of $C^\alpha_{\tau}(\PC)$ as follows.  Let $\sigma_1,\sigma_2$ be the two facets of $\mbox{st}(\tau)$ and $Se_1+Se_2$ the free $S$-module on generators $e_1,e_2$ corresponding to $\sigma_1,\sigma_2$.  Define $N(\tau)$ to be the submodule of $C^\alpha_{\tau}(\PC)$ generated by $F_1=\Lambda(\sigma_1)e_1,F_2=\Lambda(\sigma_2)e_2$, and $F_\tau=\Lambda(\mbox{st}(\tau))(e_1+e_2)$.  There is a single nontrivial syzygy among $F_1,F_2,F_\tau$ given by $L_\tau F_\tau =L_2F_1+L_1F_2$.  So $N(\tau)$ has minimal free resolution
\[
\begin{array}{cc}
& S(-\lambda(\sigma_1))\\
& \oplus\\
0\longrightarrow S(-\lambda(\mbox{st}(\tau))-\alpha(\tau)-1) \longrightarrow &  S(-\lambda(\mbox{st}(\tau)))\\
& \oplus\\
& S(-\lambda(\sigma_2))
\end{array}
\]
From Definition~\ref{defn:regdef} and the free resolution above we see that $\reg (N(\tau))=\lambda(\mbox{st}(\tau))+\alpha(\tau)=\lambda(\tau)-1$.

Now we show $\mbox{codim}(C^\alpha_{\tau}(\PC)/N(\tau))\ge 2$.  It suffices to show that $(C^\alpha_{\tau}(\PC))_P=N(\tau)_P$ for every prime of codimension one.  Since $S$ is a UFD, primes of codimension one are principle, generated by a single irreducible polynomial.  If $P\neq \langle l_\gamma \rangle$ for any $\gamma\in (\mbox{st}(\tau))_2$ then
\[
(C^\alpha_{\tau}(\PC))_P=N(\tau)_P=S^2_P.
\]
If $P=\langle l_\gamma \rangle$ for some $\gamma \in \partial^0(\mbox{st}(\tau))$, then
\[
(C^\alpha_{\tau}(\PC))_P=N(\tau)_P=l^{\alpha(\gamma)+1}_\gamma S_P \oplus S_P.
\]
if $\mbox{aff}(\gamma)$ meets only one face $\gamma\in(\mbox{st}(\tau))_2$ or
\[
(C^\alpha_{\tau}(\PC))_P=N(\tau)_P=l^{\alpha(\gamma)+1}_\gamma S_P \oplus l^{\alpha(\gamma)+1}_\gamma S_P
\]
If $\mbox{aff}(\gamma)$ meets both $\sigma_1$ and $\sigma_2$ in a codimension one face. If $P=\langle l_\tau \rangle$, then
\[
(C^\alpha_{\tau}(\PC))_P=N(\tau)_P=(C^{\alpha(\tau)}(\mbox{st}(\tau)))_P.
\]
$\mbox{pd}(C^\alpha_{\tau}(\PC))\le 1$ follows by Corollary~\ref{cor:SplinePD}, because we assumed $\PC\subset\R^3$.  Since $\mbox{codim}(C^\alpha_{\tau}(\PC)/N(\tau))\ge 2$,
\[
\reg (C^\alpha_\tau(\PC))\le \reg( N(\tau))=\lambda(\tau)-1
\]
follows from Proposition~\ref{prop:RegDepth}.
\end{proof}

\begin{thm}\label{thm:main2}
Let $\PC\subset\R^3$ be a pure $3$-dimensional polytopal complex which is central and set $e(\PC)=\max\{\lambda(\tau)|\tau\in\PC^0_2\}$.  Then 
\begin{enumerate}
\item $\reg (C^\alpha(\PC))\le e(\PC)-1$
\item $\wp(C^\alpha(\PC))\le e(\PC)-3$
\end{enumerate}
In particular, $HP(C^\alpha(\PC),d)=\dim_\R C^r_d(\PC)$ for $d\ge e(\PC)-2$.
\end{thm}
\begin{proof}
(1) follows by applying Theorem~\ref{thm:main1} to Proposition~\ref{prop:edgereg}.  (2) follows from (1) by Corollary~\ref{cor:SplinePD}.
\end{proof}

Example~\ref{ex:HPolytopal} indicates that the bound given in Theorem~\ref{thm:main2} can be far from optimal.  In the next section we bound $\reg (C^\alpha_\tau(\Delta))$ more precisely for $\Delta\subset\R^3$ a central simplicial complex.

\section{Simplicial Regularity Bound}\label{sec:SStar}

In this section we analyze the regularity of the ring of splines $C^\alpha_\tau(\Delta)$ vanishing outside the star of $2$-face, for $\Delta\subset\R^3$ a pure three-dimensional hereditary simplicial complex which is central.  Again we assume $\alpha(\tau)=-1$ for $\tau\in\Delta_2$ with $\mathbf{0}\notin\mbox{aff}(\tau)$, so that $C^\alpha(\Delta)$ is a graded module over the polynomial ring $S=\R[x,y,z]$.  This means that $\mbox{st}(\tau)$ has at most five $2$-faces $\gamma$ for which $\alpha(\gamma)\ge 0$ ($\alpha(\tau)\ge 0$ is required).  We prove the following theorem.

\begin{thm}\label{thm:simpedgereg}
Let $\tau\in\Delta^0_2$ be a $2$-face.  Define
\[
M(\tau)=(\alpha(\tau)+1)+\max\{(\alpha(\gamma_1)+1)+(\alpha(\gamma_2)+1)|\gamma_1\neq\gamma_2\in(\mbox{st}(\tau))_2 \}.
\]
Then $\reg (C^\alpha_\tau(\Delta))\le M(\tau)$.
\end{thm}

Before proving Theorem~\ref{thm:simpedgereg} we derive a couple of corollaries.

\begin{thm}\label{thm:main3}
Let $\Delta\subset\R^3$ be a pure $3$-dimensional hereditary simplicial complex which is central.  For $\tau\in\Delta^0_2$, let $M(\tau)$ be defined as in Theorem~\ref{thm:simpedgereg}.  Then
\begin{enumerate}
\item $\reg (C^\alpha(\Delta))\le \max\{M(\tau)|\tau\in\Delta^0_2\}$
\item $\wp (C^\alpha(\Delta))\le \max\{M(\tau)|\tau\in\Delta^0_2\}-2$
\end{enumerate}
In particular, $HP(C^\alpha(\Delta),d)=\dim_\R C^r(\Delta)_d$ for $d\ge \max\{M(\tau)|\tau\in\Delta^0_2\}-1$.
\end{thm}
\begin{proof}
(1) follows by applying Theorem~\ref{thm:main1} to Theorem~\ref{thm:simpedgereg}, (2) follows by applying Theorem~\ref{thm:HFHP} to (1).
\end{proof}

Setting $\alpha(\tau)=r$ for all $\tau\in\Delta^0_2$, we obtain

\begin{cor}\label{cor:simpreg}
Let $\Delta\subset\R^3$ be a pure $3$-dimensional hereditary simplicial complex which is central.  Then
\begin{enumerate}
\item $\reg (C^\alpha(\Delta))\le 3r+3$
\item $\wp (C^\alpha(\Delta))\le 3r+1$
\end{enumerate}
In particular, $HP(C^r(\Delta),d)=\dim_\R C^r(\Delta)_d$ for $d\ge 3r+2$.
\end{cor}

This result was obtained in the case of $C^r(\wDelta)$, for simplicial $\Delta\subset\R^2$, by Hong~\cite{HongDong} and Ibrahim and Schumaker~\cite{SuperSpline} (see Table~\ref{tbl:PostBounds} in the introduction).  Before proving Theorem~\ref{thm:main3} we set up some notation. Figure~\ref{fig:EStar} depicts our situation.  We will abuse notation and write $v_i$ both for the corresponding edge of $\mbox{st}(\tau)$ and for the vector we obtain by taking positive real multiples of this edge.

\begin{figure}[htp]
\centering
\begin{tikzpicture}[scale=.3,x={(-1.414 cm, -1.414 cm)},y={(2 cm,0 cm)},z={(0 cm, 2 cm)},>=stealth]

\tikzstyle{gry}=[fill=gray!50!white,line width=1 pt,join=round]
\tikzstyle{gryo}=[fill=gray,thick,join=round]

\node[shape=coordinate] (v1) at (4,.4,1){};
\node[shape=coordinate] (v2) at (.4,4,1){};
\node[shape=coordinate] (v4) at (1,2,-3){};
\node[shape=coordinate] (v3) at (2,1,5){};

\draw[->] (0,0,0)--node[at end, right]{$x$}(4,0,0);
\draw[->] (0,0,0)--node[at end, below]{$y$}(0,4,0);
\draw[->] (0,0,0)--node[at end, right]{$z$}(0,0,4);

\draw[->,thin] (0,0,0)--node[at end, below]{$v_1$} (4.4,.44,1.1);
\draw[->,thin] (0,0,0)--node[at end, below right]{$v_2$} (.44,4.4,1.1);
\draw[->,thin] (0,0,0)--node[at end, left]{$v_3$} (2.2,1.1,5.5);
\draw[->,thin] (0,0,0)--node[at end, left]{$v_4$} (1.1,2.2,-3.3);

\filldraw[gry] (v2)--(0,0,0)--(v4)--cycle;
\draw (7/15, 2, -2/3) node {$e_{24}$};
\filldraw[gry] (v1)--(0,0,0)--(v4)--cycle;
\draw (5/3, 4/5, -2/3) node {$e_{14}$};
\filldraw[gryo] (v1)--(0,0,0)--(v2)--cycle;
\draw (22/15, 22/15, 2/3) node {$\tau$};
\filldraw[gry] (v1)--(0,0,0)--(v3)--cycle;
\draw (2, 7/15, 2) node {$e_{13}$};
\filldraw[gry] (v2)--(0,0,0)--(v3)--cycle;
\draw (4/5, 5/3, 2) node {$e_{23}$};

\end{tikzpicture}
\caption{$\mbox{st}(\tau)$}\label{fig:EStar}
\end{figure}

Let $u_1,u_2\in S$ be the forms corresponding to the $2$-faces $e_{13},e_{23}$, let $w_1,w_2$ be the forms corresponding to the $2$-faces $e_{14},e_{24}$, and $l_\tau$ be the form corresponding to $\tau$ (for now do this without coordinates).  Let $\alpha_\tau=\alpha(\tau)+1,\alpha_1=\alpha(e_{13})+1,\alpha_2=\alpha(e_{23})+1,\beta_1=\alpha(e_{14})+1,\beta_2=\alpha(e_{24})+1$ be the exponents to appear on $l_\tau, u_1,u_2,w_1,w_2$ corresponding to the smoothness parameters specified by $\alpha$.  The following lemma is a special case of Proposition~\ref{prop:idsyz}.

\begin{lem}\label{lem:splineideal}
Let $K(\tau)=(l^{\alpha_\tau}_\tau,u^{\alpha_1}_1 u^{\alpha_2}_2,w^{\beta_1}_1w^{\beta_2}_2)$.  Then we have a graded isomorphism
\[
C^\alpha_\tau(\Delta)\cong\left\lbrace
\begin{array}{ll}
S(-\alpha_\tau-\beta_1-\beta_2)\oplus S(-\alpha_1-\alpha_2) & \mbox{if } u^{\alpha_1}_1 u^{\alpha_2}_2\in\left\langle w^{\beta_1}_1w^{\beta_2}_2,l^{\alpha_\tau}_\tau \right\rangle\\
S(-\alpha_\tau-\alpha_1-\alpha_2)\oplus S(-\beta_1-\beta_2) & \mbox{if } w^{\beta_1}_1w^{\beta_2}_2\in\left\langle u^{\alpha_1}_1 u^{\alpha_2}_2,l^{\alpha_\tau}_\tau \right\rangle\\
S(-\alpha_1-\alpha_2-\beta_1-\beta_2)\oplus S(-\alpha_\tau) & \mbox{if } l^{\alpha_\tau}_\tau\in\left\langle u^{\alpha_1}_1 u^{\alpha_2}_2,w^{\beta_1}_1w^{\beta_2}_2 \right\rangle\\
\syz( K(\tau)) & \mbox{otherwise},
\end{array}\right.
\]
where $\syz (K(\tau))$ is the module of syzygies on the ideal $K(\tau)$.
\end{lem}

\begin{proof}[Proof of Theorem~\ref{thm:simpedgereg}]

If $u^{\alpha_1}_1 u^{\alpha_2}_2\in\left\langle w^{\beta_1}_1w^{\beta_2}_2,l^{\alpha_\tau}_\tau \right\rangle$ or $w^{\beta_1}_1w^{\beta_2}_2\in\left\langle u^{\alpha_1}_1 u^{\alpha_2}_2,l^{\alpha_\tau}_\tau \right\rangle$ then $\reg(C^\alpha_\tau(\Delta))\le M(\tau)$ is clear from Lemma~\ref{lem:splineideal}.  If $l^{\alpha_\tau}_\tau\in\left\langle u^{\alpha_1}_1 u^{\alpha_2}_2,w^{\beta_1}_1w^{\beta_2}_2 \right\rangle$ then $\alpha_\tau\ge \alpha_1+\alpha_2$, $\alpha_\tau\ge \beta_1+\beta_2$, and $\reg(C^\alpha_\tau(\Delta))\le M(\tau)$ from Lemma~\ref{lem:splineideal}.  So we may assume $K(\tau)$ is minimally generated by the three given forms and $C^\alpha_\tau(\Delta)\cong \syz (K(\tau))$.  In this case $\reg(C^\alpha_\tau(\Delta))\le\reg (S/K(\tau))+2$ by two applications of Proposition~\ref{prop:regseq} (equality holds but we will not need this).  So it suffices to show that $\reg(S/K(\tau))\le M(\tau)-2$.

Four special cases are given by
\begin{enumerate}
\item $\alpha_1=\beta_1=0\implies K(\tau)=\langle l_\tau^{\alpha_\tau},u_2^{\alpha_2},w_2^{\beta_2} \rangle$
\item $\alpha_2=\beta_2=0\implies K(\tau)=\langle l_\tau^{\alpha_\tau},u_1^{\alpha_1},w_1^{\beta_1} \rangle$
\item $\alpha_1=\beta_2=0\implies K(\tau)=\langle l_\tau^{\alpha_\tau},u_2^{\alpha_2},w_1^{\beta_1} \rangle$
\item $\alpha_2=\beta_1=0\implies K(\tau)=\langle l_\tau^{\alpha_\tau},u_2^{\alpha_2},w_1^{\beta_1} \rangle$
\end{enumerate}
Since $K(\tau)$ is minimally generated by the three given forms, ~\cite[Theorem~2.7]{FatPoints} applies in cases (1) and (2).  For example, in case (1) we have
\[
\begin{array}{rl}
\reg (S/K(\tau))&=\left\lfloor\dfrac{\alpha_\tau+\alpha_2+\beta_2-3}{2}\right\rfloor\\
& \le \alpha_\tau+\alpha_2+\beta_2-2\\
& \le M(\tau)-2.
\end{array}
\]
A similar argument holds for case (2).  In cases (3) and (4), $K(\tau)$ is a complete intersection of its generators and $\reg (K(\tau))\le M(\tau)-2$ follows from the Koszul resolution.

If at most one of $\alpha_1,\alpha_2,\beta_1,\beta_2$ vanishes we show $\reg (S/K(\tau))\le M(\tau)-2$ by fitting $S/K(\tau)$ into exact sequences and using Proposition~\ref{prop:regseq}.  Let $Q=\left\langle l_\tau^{\alpha_\tau},u_1^{\alpha_1}u_2^{\alpha_2}\right\rangle$.  We have the short exact sequence
\begin{eqnarray}\label{eqn:ssbeta}
0\rightarrow \dfrac{S(-\beta_1-\beta_2)}{Q:(w^{\beta_1}_1w^{\beta_2}_2)}\xrightarrow{\cdot w^{\beta_1}_1w^{\beta_2}_2} \dfrac{S}{Q} \rightarrow \dfrac{S}{K(\tau)} \rightarrow 0
\end{eqnarray}
$Q$ is a complete intersection with $2$ generators in degrees $\alpha_\tau$ and $\alpha_1+\alpha_2$, so
\[
\reg (S/Q)=\alpha_\tau+\alpha_1+\alpha_2-2
\]
The ideal $Q$ decomposes as $Q=\langle l_\tau^{\alpha_\tau},u_1^{\alpha_1}\rangle\cap \langle l_\tau^{\alpha_\tau},u_2^{\alpha_2}\rangle$.  Then 
\[
Q:(w_1^{\beta_1}w_2^{\beta_2})=I_1\cap I_2,
\]
where
\[
\begin{array}{rl}
I_1=& \langle l^{\alpha_\tau}_\tau,u_1^{\alpha_1}\rangle:(w^{\beta_1}_1w^{\beta_2}_2)\\
I_2=& \langle l^{\alpha_\tau}_\tau,u_2^{\alpha_2}\rangle:(w^{\beta_1}_1w^{\beta_2}_2)
\end{array}
\]
Since $(l^{\alpha_\tau}_\tau,u_1^{\beta_1})$ is $(l_\tau,u_1)$-primary and $w_2\notin (l_\tau,u_1)$, 
\[
I_1=\langle l^{\alpha_\tau}_\tau,u_1^{\alpha_1}\rangle: (w^{\beta_1}_1w^{\beta_2}_2)=\langle l^{\alpha_\tau}_\tau,u_1^{\alpha_1}\rangle:w_1^{\beta_1}.
\]
Similarly,
\[
I_2=\langle l^{\alpha_\tau}_\tau,u_2^{\alpha_2}\rangle:(w^{\beta_1}_1w^{\beta_2}_2)=\langle l_\tau^{\alpha_\tau},u_2^{\alpha_2}\rangle:w_2^{\beta_2}.
\]
From Proposition~\ref{prop:DescribeI} below, if $I_1\neq S$ and $I_2\neq S$ then $I_1,I_2$ are complete intersections and
\[
\begin{array}{rl}
\reg (S/I_1) & \le \alpha_\tau+\alpha_1-\beta_1-2 \\
\reg (S/I_2) & \le \alpha_\tau+\alpha_2-\beta_2-2
\end{array}
\]
We consider four final special cases before moving on to the general case.
\begin{description}
\item[A] $w^{\beta_1}_1\in \langle l^{\alpha_\tau}_\tau,u_1^{\alpha_1}\rangle\implies I_1=S \implies Q:(w^{\beta_1}_1w^{\beta_2}_2)=I_2$
\item[B] $w^{\beta_2}_2\in \langle l^{\alpha_\tau}_\tau,u_2^{\alpha_2}\rangle\implies I_2=S \implies Q:(w^{\beta_1}_1w^{\beta_2}_2)=I_1$
\end{description}
Note that $\alpha_1=0$ falls under \textbf{A} and $\alpha_2=0$ falls under \textbf{B}.  By the exact sequence ~\eqref{eqn:ssbeta} and Proposition~\ref{prop:regseq} we have the corresponding bounds
\begin{description}
\item[A] $\reg (S/K(\tau))\le\max\{\alpha_\tau+\alpha_2+\beta_1-3,\alpha_\tau+\alpha_1+\alpha_2-2\}\le M(\tau)-2$
\item[B] $\reg (S/K(\tau))\le\max\{\alpha_\tau+\alpha_1+\beta_2-3,\alpha_\tau+\alpha_1+\alpha_2-2\}\le M(\tau)-2$
\end{description}
If we use multiplication by $u^{\alpha_1}_1u^{\alpha_2}_2$ in the exact sequence ~\eqref{eqn:ssbeta} then we have the corresponding ideals $Q'=\langle l^{\alpha_\tau}_\tau, u_1^{\alpha_1}u_2^{\alpha_2} \rangle$,$I'_1=\langle l^{\alpha_\tau}_\tau,w_1^{\beta_1}\rangle:u_1^{\alpha_1}$ and $I'_2=\langle l^{\alpha_\tau}_\tau,w_2^{\beta_2}\rangle:u_2^{\alpha_2}$.  We then have the analogous cases
\begin{description}
\item[C] $u^{\alpha_1}_1\in \langle l^{\alpha_\tau}_\tau,w_1^{\beta_1}\rangle\implies I'_1=S \implies Q':(u^{\alpha_1}_1u^{\alpha_2}_2)=I'_2$
\item[D] $u^{\alpha_2}_2\in \langle l^{\alpha_\tau}_\tau,w_2^{\beta_2}\rangle\implies I'_2=S \implies Q':(u^{\alpha_1}_1u^{\alpha_2}_2)=I'_1$
\end{description}
Note that $\beta_1=0$ falls under \textbf{C} and $\beta_2=0$ falls under \textbf{D}. The corresponding bounds are
\begin{description}
\item[C] $\reg (S/K(\tau))\le\max\{\alpha_\tau+\beta_2+\alpha_1-3,\alpha_\tau+\beta_1+\beta_2-2\}\le M(\tau)-2$
\item[D] $\reg (S/K(\tau))\le\max\{\alpha_\tau+\beta_1+\alpha_2-3,\alpha_\tau+\beta_1+\beta_2-2\}\le M(\tau)-2$.
\end{description}

We have reduced to the case where

\begin{itemize}
\item $w^{\beta_i}_i\notin \langle l^{\alpha_\tau}_\tau,u_i^{\alpha_i}\rangle$ (equivalently $\alpha_\tau+\alpha_i-\beta_i\ge 2$) for $i=1,2$
\item $u^{\alpha_i}_i\notin \langle l^{\alpha_\tau}_\tau,w_i^{\beta_i}\rangle$ (equivalently $\alpha_\tau+\beta_i-\alpha_i\ge 2$) for $i=1,2$
\item $\alpha_i\ge 1,\beta_i\ge 1$ for $i=1,2$ and $\alpha_\tau\ge 1$.
\end{itemize}

In particular, $u_1\neq w_1$ implies that the vectors $v_1,v_3,v_4$ are linearly independent and $u_2\neq w_2$ implies $v_2,v_3,v_4$ are linearly independent in Figure~\ref{fig:EStar}.  It follows that we may make a change of coordinates so that $v_1$ points along the $y$-axis, $v_2$ points along the $x$-axis, and $v_3$ points along the $z$-axis.  Applying appropriate scaling in the $x$, $y$, and positive $z$ directions, we can assume that the vector defined by $v_4$ points in the direction of $\langle 1,1,-1\rangle$.  Under this change of coordinates, $\mbox{st}(\tau)$ has four possible configurations, shown in Figure~\ref{fig:starconfigs}.  The ideal $K(\tau)$ is the same for all of these.  We have
\[
\begin{array}{rl}
l_\tau=&z\\
u_1=&x\\
u_2=&y\\
w_1=&x+z\\
w_2=&y+z
\end{array}
\]
and
\[
\begin{array}{rl}
I_1= & \langle l^{\alpha_\tau}_\tau,u_1^{\alpha_1}\rangle:w_1^{\beta_1}= \langle x^{\alpha_1},z^{\alpha_\tau}\rangle:(x+z)^{\beta_1}\\
I_2= & \langle l^{\alpha_\tau}_\tau,u_2^{\alpha_2}\rangle:w_2^{\beta_2}= \langle y^{\alpha_2},z^{\alpha_\tau}\rangle:(y+z)^{\beta_2}
\end{array}
\]
By Corollary~\ref{cor:initial3} in the next section, $\reg (S/Q)=\reg (S/(I_1\cap I_2))\le M(\tau)-\beta_1-\beta_2-1$.  By the exact sequence~\eqref{eqn:ssbeta} and Lemma~\ref{prop:regseq}, the proof is complete.
\end{proof}

\begin{figure}[htp]
\centering

\begin{tikzpicture}[scale=.65,x={(-1.414 cm, -1.414 cm)},y={(2 cm,0 cm)},z={(0 cm, 2 cm)},>=stealth]

\tikzstyle{gry}=[fill=gray!50!white,line width=1 pt,join=round]
\tikzstyle{gryo}=[fill=gray,thick,join=round]

\draw[<->] (-1.2,0,0)--node[at end, below]{$x$}(1.2,0,0);
\draw[<->] (0,-1.2,0)--node[at end, below]{$y$}(0,1.2,0);
\draw[<->] (0,0,-1)--node[at end, right]{$z$}(0,0,1.2);

\filldraw[gry] (0,1,0)--(0,0,0)--(0,0,1)--cycle;
\filldraw[gry] (1,0,0)--(0,0,0)--(0,0,1)--cycle;
\filldraw[gry] (1,0,0)--(0,0,0)--(1,1,-1)--cycle;
\filldraw[gry] (0,0,0)--(0,1,0)--(1,1,-1)--cycle;
\filldraw[gryo] (1,0,0)--(0,0,0)--(0,1,0)--cycle;

\draw (.33,.33,0) node{$\tau$};

\end{tikzpicture}
\begin{tikzpicture}[scale=.65,x={(-1.414 cm, -1.414 cm)},y={(2 cm,0 cm)},z={(0 cm, 2 cm)},>=stealth]

\tikzstyle{gry}=[fill=gray!50!white,line width=1 pt,join=round]
\tikzstyle{gryo}=[fill=gray,thick,join=round]

\filldraw[gry] (0,0,0)--(0,-1,0)--(1,1,-1)--cycle;
\filldraw[gryo] (1,0,0)--(0,0,0)--(0,-1,0)--cycle;
\draw (.33,-.38,0) node{$\tau$};

\draw[<->] (-1.2,0,0)--node[at end, below]{$x$}(1.2,0,0);
\draw[<->] (0,-1.2,0)--node[at end, below]{$y$}(0,1.2,0);
\draw[<->] (0,0,-1)--node[at end, right]{$z$}(0,0,1.2);

\filldraw[gry] (0,-1,0)--(0,0,0)--(0,0,1)--cycle;
\filldraw[gry] (1,0,0)--(0,0,0)--(0,0,1)--cycle;

\filldraw[gry] (1,0,0)--(0,0,0)--(1,1,-1)--cycle;

\end{tikzpicture}

\begin{tikzpicture}[scale=.65,x={(-1.414 cm, -1.414 cm)},y={(2 cm,0 cm)},z={(0 cm, 2 cm)},>=stealth]
\tikzstyle{gry}=[fill=gray!50!white,line width=1 pt,join=round]
\tikzstyle{gryo}=[fill=gray,thick,join=round]

\filldraw[gryo] (-1,0,0)--(0,0,0)--(0,-1,0)--cycle;
\draw (-.33,-.33,0) node{$\tau$};

\draw[<->] (0,0,-1)--node[at end, right]{$z$}(0,0,1.2);

\filldraw[gry] (0,-1,0)--(0,0,0)--(0,0,1)--cycle;
\filldraw[gry] (-1,0,0)--(0,0,0)--(0,0,1)--cycle;

\filldraw[gry] (0,0,0)--(0,-1,0)--(1,1,-1)--cycle;
\filldraw[gry] (-1,0,0)--(0,0,0)--(1,1,-1)--cycle;

\draw[<->] (-1.2,0,0)--node[at end, below]{$x$}(1.2,0,0);
\draw[<->] (0,-1.2,0)--node[at end, below]{$y$}(0,1.2,0);

\end{tikzpicture}
\begin{tikzpicture}[scale=.65,x={(-1.414 cm, -1.414 cm)},y={(2 cm,0 cm)},z={(0 cm, 2 cm)},>=stealth]

\tikzstyle{gry}=[fill=gray!50!white,line width=1 pt,join=round]
\tikzstyle{gryo}=[fill=gray,thick,join=round]

\filldraw[gry] (-1,0,0)--(0,0,0)--(1,1,-1)--cycle;
\filldraw[gryo] (-1,0,0)--(0,0,0)--(0,1,0)--cycle;
\draw (-.65,.23,0) node{$\tau$};

\draw[<->] (-1.2,0,0)--node[at end, below]{$x$}(1.2,0,0);
\draw[<->] (0,-1.2,0)--node[at end, below]{$y$}(0,1.2,0);
\draw[<->] (0,0,-1)--node[at end, right]{$z$}(0,0,1.2);

\filldraw[gry] (-1,0,0)--(0,0,0)--(0,0,1)--cycle;
\filldraw[gry] (0,0,0)--(0,1,0)--(1,1,-1)--cycle;
\filldraw[gry] (0,1,0)--(0,0,0)--(0,0,1)--cycle;
\end{tikzpicture}
\caption{Possible configurations for generic $\mbox{st}(\tau)$}\label{fig:starconfigs}
\end{figure}

\subsection{Intersection of colon ideals}
Let
\[
\begin{array}{rl}
I_1= & \langle x^{\alpha_1},z^{\alpha_\tau}\rangle:(x+z)^{\beta_1}\\
I_2= & \langle y^{\alpha_2},z^{\alpha_\tau}\rangle:(y+z)^{\beta_2}
\end{array}
\]
In~\cite{Stefan}, Tohaneanu and Minac compute the Hilbert function of the ideal (up to change of coordinates)
\[
\langle x^{r+1},(x+z)^{r+1}\rangle:z^{r+1} \cap \langle y^{r+1},(y+z)^{r+1}\rangle:z^{r+1}.
\]
It is not so obvious how to apply their methods directly to the ideal $I_1\cap I_2$.  Building on their work, however, we show how to construct enough of the initial ideal (with respect to the lexicographic order) of $I_1+I_2$ to give a fairly tight bound on the socle degree of $S/I_1+I_2$.  The methods synthesize descriptions of such ideals in terms of linear and commutative algebra.

We first compute the initial ideal of
\[
I=I(p,q,r)=\langle s^p,t^q \rangle:(s+t)^r
\]
in the ring $R=k[s,t]$ with standard lexicographic order.  We assume $I\neq R$, so $(s+t)^r\notin \langle s^p,t^q \rangle$.  This is equivalent to requiring $p+q-r\ge 2$.

\begin{prop}\label{prop:DescribeI}
Let $I=I(p,q,r)\subset R$ be as above, with $p+q-r\ge 2$.  Then $I$ is a complete intersection generated by two polynomials of
\begin{enumerate}
\item degrees $a=\min\{p,q-r\},b=\max\{p,q-r\}$ if $p+r-q\le 1$.
\item degrees $a=\min\{q,p-r\},b=\max\{q,p-r\}$ if $q+r-p\le 1$.
\item degrees
\[
a=\left\lfloor \dfrac{p+q-r}{2} \right\rfloor, b=\left\lceil \dfrac{p+q-r}{2} \right\rceil
\]
if $p+r-q\ge 2$ and $q+r-p\ge 2$.
\end{enumerate}
\end{prop}
\begin{proof}
\textbf{$p+r-q\le 1:$} In this case $t^q\in\left\langle s^p, (s+t)^r \right\rangle$.  Let
\[
t^q=fs^p+g(s+t)^r
\]
for some polynomials $f,g\in R$, where $g$ has no term divisible by $s^p$.  It is immediate that
\[
\langle s^p,t^q \rangle=\langle s^p,g(s+t)^r \rangle
\]
and
\[
I=\langle s^p,t^q \rangle:(s+t)^r=\langle s^p, g \rangle.
\]
The polynomial $g$ is not divisible by $s$ since it has a term which is a constant multiple of $t^{q-r}$.  It follows that $s^p$ ang $g$ are relatively prime and $I$ is a complete intersection.  Since $g$ has degree $q-r$, (1) is proved.

\textbf{$q+r-p\le 1:$} The argument is identical to the previous case.

\textbf{$p+r-q\ge 2$ and $q+r-p\ge 2$:}  Let
\[
T=T(p,q,r)=\langle s^p,t^q,(s+t)^r \rangle.
\]
Since we assume $p+q-r\ge 2$ as well, $T$ is minimally generated by the three given generators.  We describe $I$ in terms of the minimal free resolution of the ideal $T$.  Set $a=\left\lfloor\dfrac{p+q-r}{2}\right\rfloor$ and $b=\left\lceil\dfrac{p+q-r}{2}\right\rceil$.  The assumption $p+q-r\ge 2$ guarantees that $a\ge 1$.  $T$ is a codimension two Cohen-Macaulay ideal with Hilbert-Burch resolution of the form below~\cite[Theorem~2.7]{FatPoints}
\[
0\rightarrow R(-a-r)\oplus R(-b-r) \xrightarrow{\phi} R(-p)\oplus R(-q) \oplus R(-r) \rightarrow T
\]
where
\[
\phi=\left(
\begin{array}{cc}
A & D\\
B & E\\
C & F
\end{array}
\right)
\]
is a matrix of forms satisfying $BF-EC=s^p,AF-DC=t^q,BF-EC=(s+t)^r$.
It follows that the module of syzygies on $T$ has two generators, corresponding to the relations
\[
A s^p + B t^q + C (s+t)^r=0
\]
and
\[
D s^p + E t^q + F (s+t)^r=0.
\]
In terms of the entries of the matrix $\phi$ we may write
\[
I=(C,F)
\]
where $\mbox{deg}(C)=a,\mbox{deg}(F)=b$, and $a+b=p+q-r$.  Since $AF-DC=t^q$ and $BF-EC=(s+t)^r$, any common factor of $C$ and $F$ would give a common factor of $t$ and $(s+t)$, so $C$ and $F$ are relatively prime.  So $I$ is a complete intersection of the required degrees.
\end{proof}

As an immediate corollary we have the following lemma.
\begin{cor}\label{cor:hilbert}
With $I=I(p,q,r)$ as above, minimally generated by two forms of degree $a\le b$, we have
\[
HF(I,d)=\binom{d+1-a}{1}+\binom{d+1-b}{1}-\binom{d+1-a-b}{1}.
\]
\end{cor}
\begin{proof}
From Proposition~\ref{prop:DescribeI}, $I$ is a complete intersection of polynomials $C,F$ with $\mbox{deg}(C)=a,\mbox{deg}(F)=b$.  So $I$ has minimal resolution of the form
\[
0\rightarrow R(-a-b) \rightarrow R(-a)\oplus R(-b) \rightarrow I \rightarrow 0.
\]
The result follows from the additivity of Hilbert functions across exact sequences.
\end{proof}

Given a Hilbert function $H(I,d)$, let $L_d$ be the vector space spanned by the $H(I,d)$ greatest monomials of degree $d$ with respect to lex order.  Then the direct sum
\[
L=\bigoplus\limits_{d=0}^\infty L_d
\]
is an ideal, known as the \textit{lex-segment} ideal for the Hilbert function $H(d)$ \cite[Proposition 2.21]{Comb}.  Since two \textit{generic} forms in $\R[x,y]$ of degrees $a\le b$ form a complete intersection, the ideal they generate has the same Hilbert function as $I$.  Denote by $L(a,b)$ the corresponding lex-segment ideal.  We will show that $L(a,b)$ is the initial ideal of $I(p,q,r)$.

To prove this we will use a matrix condition on the coefficients of a form $f$ of degree $d$ which distinguishes when $f\in I$.  From Corollary~\ref{cor:hilbert}, $I_d=R_d$ for $d\ge a+b-1$.  Since $a+b=p+q-r$, this matrix condition we derive will be nontrivial for $1\le d < p+q-r-1$.  Suppose
\[
f=\sum\limits_{i+j=d}a_{i,j}s^it^j
\]
satisfies $f\in I_d$.  Then by definition we have
\[
f(s+t)^r\in (s^p,t^q)
\]
Since the ideal on the right is a monomial ideal, $f\in I\iff$ every monomial of $f(s+t)^r$ is divisible by either $s^p$ or $t^q$.  Expanding $(s+t)^r$ and multiplying by $f$ gives
\[
f(s+t)^r=\sum\limits_{i+j=d}\sum\limits_{m+n=r}\binom{r}{m}a_{ij}s^{m+i}t^{n+j}
\]
Setting $m+i=u$ and $n+j=v$ gives
\[
\sum\limits_{u+v=d+r}s^u t^v\left(\sum\limits_{m+i=u} \binom{r}{m}a_{ij}\right).
\]
$f\in I$ iff the only nonzero coefficients in this expression occur when $u\ge p$ or $v\ge q$.  Since $v=d+r-u$, $v\ge q$ is equivalent to $u\le d+r-q$.  So $f\in I$ iff for $u=d+r-q+1,\ldots,p-1$ we have the condition
\[
\sum\limits_{m+i=u} \binom{r}{m}a_{i,d-i}=0.
\]
Here we follow the convention that $\binom{A}{B}=0$ when $B<0$ or $B>A$.  These fit together into the following matrix condition on the coefficients of $f$:
\[
\begin{pmatrix}
\binom{r}{d+r-q+1} & \binom{r}{d+r-q} & \binom{r}{d+r-q-1} & \cdots & \binom{r}{r-q+1}\\
\binom{r}{d+r-q+2} & \binom{r}{d+r-q+1} & \binom{r}{d+r-q} & \cdots & \binom{r}{r-q+2}\\
\vdots & \vdots & \vdots & \ddots & \vdots \\
\binom{r}{p} & \binom{r+1}{p-1} & \binom{r+1}{p-4} & \cdots & \binom{r}{p-d-2}\\ 
\binom{r}{p-1} & \binom{r}{p-2} & \binom{r}{p-3} & \cdots & \binom{r}{p-d-1}
\end{pmatrix}
\cdot\begin{pmatrix}
a_{0,d}\\
a_{1,d-1}\\
\vdots\\
a_{d-1,1}\\
a_{d,0}
\end{pmatrix}
=0.
\]
Denote the $(p+q-r-d-1)\times (d+1)$ matrix on the left by $M(p,q,r,d)$.  $M(p,q,r,d)$ has entries
\[
M(p,q,r,d)_{i,j}=\binom{r}{d+r-q+1+i-j},
\]
where $i=0,\ldots,\min\{p-1,p+q-r-d-2\}$ and $j=0,\ldots,d$.  With this choice of indexing, column $c_j$ of $M(p,q,r,d)$ corresponds to the coefficient $a_{j,d-j}$.  The following lemma is fundamental for understanding $M(p,q,r,d)$.

\begin{lem}\label{lem:Schur}
Let $\mu=(\mu_0 \ge \ldots \mu_k \ge 1)$ be a partition with $k+1$ parts so that $r\ge \mu_0$.  Let $N(\mu)$ be the square matrix with entries
\[
N(\mu)_{ij}=\binom{r}{\mu_j+i-j}
\]
for $i=0,\ldots,k$, $j=0,\ldots,k$.  Then $N(\mu)$ has nonzero determinant.
\end{lem}
\begin{proof}
This observation is made in~\cite[\S~3.1]{Stefan}, where it is noted that determinants of such matrices play a role in the representation theory of the special linear group $SL(V)$, where $V$ is an $r$-dimensional vector space.  In particular, if $\lambda=\mu'$, the conjugate partition to $\mu$, $\mbox{det }N(\mu)$ is the dimension of the Weyl module $\mathbb{S}_\lambda V$, which is a nontrivial irreducible representation of $SL(V)$.  More explicitly, $\mbox{det }N(\mu)=s_\lambda(1,\ldots,1)$, where $s_\lambda(x_1,\ldots,x_r)$ is the Schur polynomial in $r$ variables of the partition $\lambda=\mu'$.  In particular, $N(\mu)$ has nonzero determinant.  See~\cite[\S~6.1]{Rep} and ~\cite[Appendix~A.1]{Rep} for more details.
\end{proof}

\begin{cor}\label{cor:schurinv}
Let $M=M(p,q,r,d)$ be the $(p+q-r-d-1)\times (d+1)$ matrix defined as above, $M_{s,t}$ a \textit{nonzero} entry of $M$, and $k$ a nonnegative integer so that $s+k\le p+q-r-d-1$ and $t+k\le d+1$.  Then 
\begin{enumerate}
\item The $(k+1)\times(k+1)$ submatrix of $M$ formed by the entries $\{M_{i,j}| s\le i \le s+k,t\le j\le t+k\}$ is invertible.
\item The rank of $M$ is the minimum of the number of nonzero rows of $M$ and the number of nonzero columns of $M$.
\end{enumerate}
\end{cor}
\begin{proof}
(1) The submatrix of $M=M(p,q,r,d)$ above has entries
\[
\binom{r}{d+r-q+1+s-t+i-j}
\]
for $i=0,\ldots,k$,$j=0,\ldots,k$.  Since we assume $M_{s,t}\neq 0$, $d+r-q+1+s-t\le r$ and the first statement follows from Lemma~\ref{lem:Schur} by taking $\mu_0=\cdots=\mu_k=d+r-q+1+s-t$.  

(2) Observe that either $M_{0,0}\neq 0$ or, if $M_{0,0}=0$, then the first entry $M_{j,0}$ ($j>0$) which is nonzero is equal to $1$.  The last entry in the first column is $\binom{r}{p-1}\ge 1$ (we assumed $p\ge 1$), so there is at least one nonzero entry in the first column of $M$.  If the row of $M$ with index $i$ is nonzero, every row with index $\ge i$ is also nonzero.  If the column of $M$ with index $j$ is zero, every column with index $\ge j$ is also zero.  Now the second statement follows by taking the largest square submatrix of $M$ whose upper left corner is the first nonzero entry of the first column of $M$.  This is a $k\times k$ submatrix of $M$ where $k$ is the minimum of the number of nonzero rows of $M$ and the number of nonzero columns of $M$. By the first statement, this submatrix is invertible, and from the earlier observations $k$ must be the rank of $M$.
\end{proof}

\begin{lem}\label{lem:initial}
The initial ideal of $I=I(p,q,r)$ is the lex-segment ideal $L(a,b)$, where $a\le b$ are the degrees of the generators of $I$.
\end{lem}
\begin{proof}
For fixed degree $d$, let $\row$ be the rank of $M(p,q,r,d)$.  By definition,
\[
HF(I,d)=\dim\mbox{ker} M(p,q,r,d),
\]
hence $HF(I,d)=d+1-\row$.  From the submatrix constructed to prove part (2) of Corollary~\ref{cor:schurinv}, the first $\row$ columns of $M$ are linearly independent.  It follows that for any column $c_l$ of $M(p,q,r,d)$ with $\row-1\le l \le d$, there is a unique (up to scaling) relation
\[
\left(\sum\limits_{i=0}^{\row-1}a_{i,d-i}c_i\right) + a_{l,d-l}c_l=0,
\]
where $a_{l,d-l}\neq 0$.  This gives rise to the polynomial $f=\sum_{i=0}^{p+q-r-d-2} a_{i,j}s^it^{d-i}+a_{l,d-l}s^lt^{d-l}\in I$ with leading monomial $s^lt^{d-l}$.  These monomials are the largest $d+1-\row$ monomials of degree $d$ with respect to lex ordering, so the result follows.
\end{proof}

\begin{cor}\label{cor:initial1}
Let $I=(s^p,t^q):(s+t)^r$, generated in degrees $a$ and $b$, with $a\le b$.  The initial ideal of $I$ with respect to the standard lexicographic order is
\[
L(a,b)=(s^a,s^{a-1}t^{b-a+1},s^{a-2}t^{b-a+3},\ldots,s^{a-i}t^{b-a+2i-1},\ldots,t^{a+b-1}).
\]
\end{cor}

\begin{proof}
By Lemma~\ref{lem:initial} it suffices to show that the lex-segment ideal $L(a,b)$ has the form above.  The Hilbert function of a complete intersection $I$ generated in degrees $a$ and $b$ is
\[
HF(I,d)=\binom{d+1-a}{1}+\binom{d+1-b}{1}-\binom{d+1-a-b}{1}.
\]
More explicitly, we have
\[
HF(I,d)=\left\lbrace
\begin{array}{ll}
0 & \mbox{for } 0\le d< a\\
d+1-a & \mbox{for } a\le d< b\\
2d+2-(a+b) & \mbox{for } b\le d \le a+b-1\\
d+1 & \mbox{for }  d> a+b-1
\end{array}
\right.
\]
Recall $L(a,b)_d$ is the vector space spanned by the $HF(I,d)$ greatest monomials of degree $d$ with respect to lex order.  If $a\le d <b$, then the $d+1-a$ greatest monomials are $s^d,\ldots,s^a$.  These are all divisible by $s^a$.  If $b\le d \le a+b-1$, the $2d+2-(a+b)$ greatest monomials are $\{s^{d-i}t^i|i=0,\ldots,2d-(a+b)+1\}$.  If $i\le d-a$, $s^{d-i}t^i$ is divisible by $s^a$.  If $d-a\le i \le 2d+1-(a+b)$, then $s^{d-i}t^i=s^{a-j}t^{d-a+j}=s^{a-j}t^{b-a+(d-b+j)}$, where $j=1,\ldots,d-b+1$.  This is divisible by $s^{a-j}t^{b-a+2j-1}$, which proves the corollary.
\end{proof}

\begin{remark}
Conca and Valla~\cite{CanonicalHB} parametrize of all ideals in two variables with a given initial ideal.  Using this, one can show that the lex-segment ideal $L(a,b)$ is the initial ideal of any ideal generated by two \textit{generic} forms of degree $a$ and $b$.  Here \textit{generic} means there are certain polynomials in the coefficients of the forms that must not vanish (the condition is not equivalent to the two forms being relatively prime).  Lemma~\ref{lem:initial} can be viewed as a proof that the ideal $I$, which is generated by two forms, is generic in this sense.
\end{remark}

\begin{prop}\label{prop:initial2}
Set $R=k[x,y]$, $S=k[x,y,z]$ both with standard lexicographic orders.  For positive integers $a\le b$, $c\le d$, let $J_1,J_2\subset R=k[s,t]$ be ideals satisfying $\mbox{in}(J_1)=L(a,b)$ and $\mbox{in}(J_2)=L(c,d)$, respectively.
Let $S=k[x,y,z]$ and define ring maps $i_1,i_2:R\rightarrow S$ by $i_1(s)=x,i_1(t)=z$ and $i_2(s)=y,i_2(t)=z$.  Set $I_1=i_1(J_1)S$, $I_2=i_2(J_2)S$, and $N=\max\{a+d-1,b+c-1\}$.  Then
\[
(I_1+I_2)_N=S_N
\]
\end{prop}
\begin{proof}
It suffices to show that $(\mbox{in}(I_1)+\mbox{in}(I_2))_N=S_N$.  We have
\[
\begin{array}{rl}
\mbox{in}(I_1)= & \langle x^a \rangle +\langle x^{a-i}z^{b-a+2i-1}|i=1,\ldots,a\rangle\\
\mbox{in}(I_2)= & \langle y^c \rangle +\langle y^{c-j}z^{d-c+2j-1}|j=1,\ldots,c\rangle
\end{array}
\]
Let $m=x^iy^jz^k$ be a monomial of $S$ with degree $N$.  We claim $m\in \mbox{in}(I_1)+\mbox{in}(I_2)$.  If $i\ge a$ or $j\ge c$ then $x^a|m$ or $y^c|m$ and we are done.  So set $i=a-s,j=c-t$, where $1\le s \le a,1\le t \le c$.  If $s\le t$ then $a+c-s-t+(b-a+2s-1)=b+c-1+s-t\le N$.  So $k=N-(a+c-s-t)\ge b-a+2s-1$ and $x^iy^jz^k\in \mbox{in}(I_1)$.  If $t\le s$ then $a+c-s-t+(d-c+2t-1)=a+d-s+t-1\le N$.  So $k=N-(a+c-s-t)\ge d-c+2t-1$ and $x^iy^jz^k\in \mbox{in}(I_2)$.
\end{proof}

\begin{cor}\label{cor:initial3}
Let
\[
\begin{array}{rl}
I_1= & \langle x^{\alpha_1},z^{\alpha_\tau}\rangle:(x+z)^{\beta_1}\\
I_2= & \langle y^{\alpha_2},z^{\alpha_\tau}\rangle:(y+z)^{\beta_2}
\end{array}
\]
where $\alpha_i+\alpha_\tau-\beta_i\ge 2$ and $\beta_i+\alpha_\tau-\alpha_i\ge 2$ for $i=1,2$.  Also assume $\alpha_i\ge 1,\beta_i\ge 1$ for $i=1,2$ and $\alpha_\tau\ge 1$.  Let 
\[
\begin{array}{rl}
M(\tau)= & \alpha_\tau+\max\{\alpha_1+\alpha_2,\alpha_1+\beta_1,\alpha_1+\beta_2,\alpha_2+\beta_1, \vphantom{\}} \\  &\vphantom{\{}\alpha_2+\beta_2,\beta_1+\beta_2,\alpha_\tau+\alpha_1,\alpha_\tau+\alpha_2,\alpha_\tau+\beta_1,\alpha_\tau+\beta_2\}
\end{array}
\]
as in the statement of Theorem~\ref{thm:simpedgereg}.  Then 
\[
\reg \left(\dfrac{S}{I_1\cap I_2}\right)\le M(\tau)-\beta_1-\beta_2-1
\]
\end{cor}
\begin{proof}
We use the short exact sequence
\[
0\rightarrow \dfrac{S}{I_1\cap I_2} \rightarrow \dfrac{S}{I_1}\oplus\dfrac{S}{I_2} \rightarrow \dfrac{S}{I_1+I_2} \rightarrow 0
\]
and Proposition~\ref{prop:regseq}.  From Proposition~\ref{prop:DescribeI}, $\reg (S/I_1)=\alpha_1+\alpha_\tau-\beta_1-2\le  M(\tau)-\beta_1-\beta_2-1$ and $\reg (S/I_2)=\alpha_2+\alpha_\tau-\beta_2-2\le  M(\tau)-\beta_1-\beta_2-1$.  We show $\reg (S/(I_1+I_2))\le  M(\tau)-\beta_1-\beta_2-2$; then we are done by Proposition~\ref{prop:regseq}.  Equivalently, we show $(I_1+I_2)_d=S_d$ for $d=M(\tau)-\beta_1-\beta_2-1$.  Let $I_1,I_2$ be generated in degrees $a\le b,c\le d$ respectively.  By Lemma~\ref{lem:initial} and Proposition~\ref{prop:initial2},$(I_1+I_2)_d=S_d$ for $d\ge\max\{a+d-1,b+c-1\}$.  So we need to show that $\max\{a+d,b+c\}\le M(\tau)-\beta_1-\beta_2$.  We consider $4$ cases.
\begin{enumerate}
\item $\beta_i+\alpha_i-\alpha_\tau\ge 2$ for $i=1,2$.
\item $\beta_1+\alpha_1-\alpha_\tau\le 1$ and $\beta_2+\alpha_2-\alpha_\tau\ge 2$.
\item $\beta_2+\alpha_2-\alpha_\tau\ge 2$ and $\beta_2+\alpha_2-\alpha_\tau\le 1$.
\item $\beta_i+\alpha_i-\alpha_\tau\le 1$ for $i=1,2$.
\end{enumerate}

\textbf{Case 1:}   By Proposition~\ref{prop:DescribeI}, $b-a\le 1$ and $d-c\le 1$.  Suppose $b<d$.  Then $a\le c$, so $b+c< d+c$ and $a+d\le c+d$, where $c+d=\alpha_2+\alpha_\tau-\beta_2\le M(\tau)-\beta_1-\beta_2$.  Similarly if $d<b$ then $b+c\le b+a$ and $a+d<a+b$, where $a+b=\alpha_1+\alpha_\tau-\beta_1\le M(\tau)-\beta_1-\beta_2$.  If $b=d$ then $a+d=a+b\le M(\tau)-\beta_1-\beta_2$ and $b+c=d+c\le M(\tau)-\beta_1-\beta_2$.  Hence $\max\{a+d-1,b+c-1\}\le M(\tau)-\beta_1-\beta_2$.

\textbf{Case 2:}  By Proposition~\ref{prop:DescribeI}, $a=\min\{\alpha_1,\alpha_\tau-\beta_1\}$ and $b=\max\{\alpha_1,\alpha_\tau-\beta_1\}$.  Since $\alpha_1\le \alpha_\tau-\beta_1+1$ by assumption, $a\le \alpha_\tau-\beta_1$ and $b\le \alpha_\tau-\beta_1+1$.  By Proposition~\ref{prop:DescribeI},
\[
\begin{array}{rl}
c=&\left\lfloor \dfrac{\alpha_2+\alpha_\tau-\beta_2}{2} \right\rfloor \\[10 pt]
d=&\left\lceil \dfrac{\alpha_2+\alpha_\tau-\beta_2}{2} \right\rceil.
\end{array}
\]
Hence
\[
\max\{a+d,b+c\}\le \alpha_\tau-\beta_1+1+\left\lfloor \dfrac{\alpha_2+\alpha_\tau-\beta_2}{2} \right\rfloor.
\]
$\alpha_2+\alpha_\tau-\beta_2\ge 2$, so
\[
\left\lfloor\dfrac{\alpha_2+\alpha_\tau-\beta_2}{2}\right\rfloor\le \alpha_2+\alpha_\tau-\beta_2-1.
\]
It follows that
\[
\begin{array}{rl}
\max\{a+d,b+c\} &\le \alpha_\tau+\alpha_2+\alpha_\tau-\beta_1-\beta_2\\
&\le M(\tau)-\beta_1-\beta_2.
\end{array}
\]

\textbf{Case 3:} By arguing exactly as in Case $2$ we obtain
\[
\begin{array}{rl}
\max\{a+d,b+c\} & \le \alpha_\tau+\alpha_1+\alpha_\tau-\beta_1-\beta_2\\
& \le M(\tau)-\beta_1-\beta_2.
\end{array}
\]

\textbf{Case 4:} By Proposition~\ref{prop:DescribeI}, $I_1$ is generated in degrees $a=\min\{\alpha_1,\alpha_\tau-\beta_1\},b=\max\{\alpha_1,\alpha_\tau-\beta_1\}$ and $I_2$ is generated in degrees $c=\min\{\alpha_2,\alpha_\tau-\beta_2\},d=\max\{\alpha_2,\alpha_\tau-\beta_2\}$.  We have $\alpha_1\le \alpha_\tau-\beta_1+1$ and $\alpha_2\le \alpha_\tau-\beta_2+1$ by assumption.  It follows that $a\le \alpha_\tau-\beta_1,b\le \alpha_\tau-\beta_1+1$ and $c\le\alpha_\tau-\beta_2,d\le \alpha_\tau-\beta_2+1$.  So 
\[
\begin{array}{rl}
\max\{a+d,b+c\} & \le \alpha_\tau+\alpha_\tau +1-\beta_1-\beta_2\\
& \le M(\tau)-\beta_1-\beta_2.
\end{array}
\]
The final inequality follows since we assumed $\alpha_1,\alpha_2,\beta_1,\beta_2,\alpha_\tau$ are all at least $1$.
\end{proof}

\section{Examples}\label{sec:examples}

We give several examples to illustrate both how the bounds in Theorems~\ref{thm:main2} and~\ref{thm:main3} may be used and how well they approximate the actual regularity of the spline algebra.  These examples also elucidate a difference between \textit{complete} central complexes $\PC$ (in which the intersection of all facets of $\PC$ is an \textit{interior} face of $\PC$) and central complexes which are not complete.  This difference is key to Conjecture~\ref{conj:c1} in the following section.

\begin{exm}\label{ex:SchlegelCube}
In this example we apply Theorem~\ref{thm:main2} to bound the regularity of $C^\alpha(\PC)$ where boundary vanishing is imposed.  Consider the two dimensional polytopal complex $\QC$ in Figure~\ref{fig:SEVL} with five faces, eight interior edges, and four interior vertices.
\begin{figure}[htp]
\includegraphics[scale=.5]{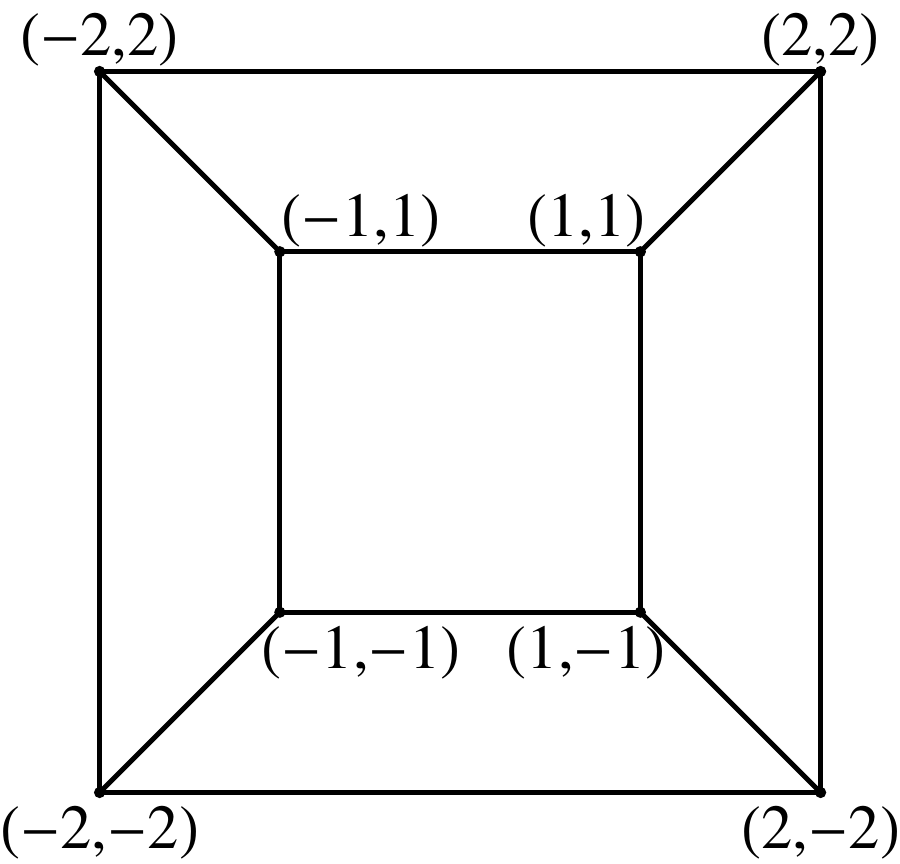}
\caption{$\QC$}\label{fig:SEVL}
\end{figure}
Impose vanishing of order $r$ along interior codimension one edges and vanishing of order $s$ along boundary codimension one faces.  The following Hilbert polynomials are computed in~\cite[Example~8.5]{Assprimes}.  If $s=-1$, then
\[
\begin{array}{rl}
HP(C^\alpha(\wPC),d)=& \frac{5}{2}  d^2+\left(-8 r-\frac{1}{2}\right)d \\
 &-4 \left\lfloor \frac{3
   r}{2}\right\rfloor ^2+12 r \left\lfloor \frac{3 r}{2}\right\rfloor
   -r^2+4 r+2
\end{array}
\]
By Theorem~\ref{thm:main2}, $\reg C^r(\widehat{\QC})\le 6(r+1)-1$ and $HP(C^r(\widehat{\QC}),d)=\dim C^\alpha_d(\QC)$ for $d\ge 6(r+1)-2$.  We compare the regularity bound $6(r+1)-1$ with $\reg C^r(\widehat{\QC})$ as computed in Macaulay2 in Table~\ref{tbl:Reg1}.  $\reg C^r(\widehat{\QC})$ appears to have alternating differences of $1$ and $3$ and grows roughly as $2(r+1)+1$.  In fact $\reg C^r(\widehat{\QC})$ appears to agree with the regularity of $r$-splines on the complex from Example~\ref{ex:HPolytopal}.

\begin{table}[htp]
\begin{center}

\begin{tabular}{c|cccccccccc}
$r$ & 0 & 1 & 2 & 3 & 4 & 5 & 6 & 7 & 8 & 9\\
\hline
$6(r+1)-1$ & 5 & 11 & 17 & 23 & 29 & 35 & 41 & 47 & 53 & 59 \\
$\reg(C^r(\widehat{\QC}))$ & 3 & 4 & 7 & 8 & 11 & 12 & 15 & 16 & 19 & 20
\end{tabular}

\end{center}
\caption{}\label{tbl:Reg1}
\end{table}

Now suppose that vanishing of degree $s\ge 0$ is imposed along $\partial\PC$.  Then
\[
\begin{array}{rl}
HP(C^\alpha(\wPC),d)= & \frac{5}{2} d^2+\left(-8 r-4 s-\frac{9}{2}\right) d \\
&-3 \left\lfloor \frac{2(r+s)}{3}\right\rfloor ^2+4 r \left\lfloor \frac{2
   (r+s)}{3}\right\rfloor +4 s \left\lfloor \frac{2 (r+s)}{3}\right\rfloor
   -\left\lfloor \frac{2 (r+s)}{3}\right\rfloor \\
   & - 4 \left\lfloor \frac{r}{2}\right\rfloor ^2-4 \left\lfloor \frac{3 r}{2}\right\rfloor
   ^2+4 r \left\lfloor \frac{r}{2}\right\rfloor +12 r \left\lfloor \frac{3
   r}{2}\right\rfloor\\
   & -5 r^2+4 r s+8 r+4 s+4.
\end{array}
\]
This formula is correct when $r,s$ are not too small; for instance if $r=3$ and $s=0$, the above formula has constant term $81$ while the actual constant, according to Macaulay2, is $87$.  By Theorem~\ref{thm:main2}, 
\[
\reg (C^\alpha(\widehat{\QC}))\le \max\{6(r+1)+(s+1),5(r+1)+2(s+1)\}-1
\]
and $HP(C^\alpha(\wPC),d)=\dim C^\alpha_d(\PC)$ for 
\[
d\ge \max\{6(r+1)+(s+1),5(r+1)+2(s+1)\}-2.
\]
A comparison of the bound on $\reg (C^\alpha(\widehat{\QC}))$ and its actual value computed in Macau-lay2 appears in Table~\ref{tbl:Reg2} for $r,s\le 5$.

\begin{table}[htp]
\begin{center}
\begin{tabular}{c|ccccc}
\multicolumn{6}{c}{$\max\{6(r+1)+(s+1),5(r+1)+2(s+1)\}-1$}\\
\hline
& $s=0$ & $s=1$ & $s=2$ & $s=3$ & $s=4$\\
\hline
$r=0$ & 7 & 8 & 10 & 12 & 14 \\
$r=1$ & 13 & 14 & 15 & 17 & 19 \\
$r=2$ & 19 & 20 & 21 & 22 & 24 \\
$r=3$ & 25 & 26 & 27 & 28 & 29 \\
$r=4$ & 31 & 32 & 33 & 34 & 35 
\end{tabular}
\medskip
\begin{tabular}{c|ccccc}
\multicolumn{6}{c}{$\reg(C^\alpha(\widehat{\QC}))$}\\
\hline
& $s=0$ & $s=1$ & $s=2$ & $s=3$ & $s=4$\\
\hline
$r=0$ & 4 & 4 & 5 & 6 & 7 \\
$r=1$ & 4 & 5 & 6 & 8 & 9 \\
$r=2$ & 7 & 8 & 8 & 9 & 10 \\
$r=3$ & 8 & 8 & 9 & 10 & 12 \\
$r=4$ & 11 & 11 & 12 & 12 & 13 \\
\end{tabular}
\end{center}
\caption{}\label{tbl:Reg2}
\end{table}
\end{exm}

\begin{exm}\label{ex:LargeFace}
We now give an example which has very different behavior from Example~\ref{ex:HighGen}.  Consider the two-dimensional polytopal complex $\QC$ formed by placing a regular (or almost regular) $n$-gon inside of a scaled copy of itself and connecting corresponding vertices by edges.  $\QC$ has one facet with $n$ edges and $n$ quadrilateral facets.  An example for $n=10$ is shown in Figure~\ref{fig:TenSpoke}.  We may or may not perturb the vertices so that the affine spans of the edges between the inner and outer $n$-gons do not all meet at the origin.  This does not appear to have much effect on regularity, although it does change the constant term of $HP(C^r(\widehat{\QC}),d)$.

\begin{figure}[htp]
\includegraphics[scale=.4]{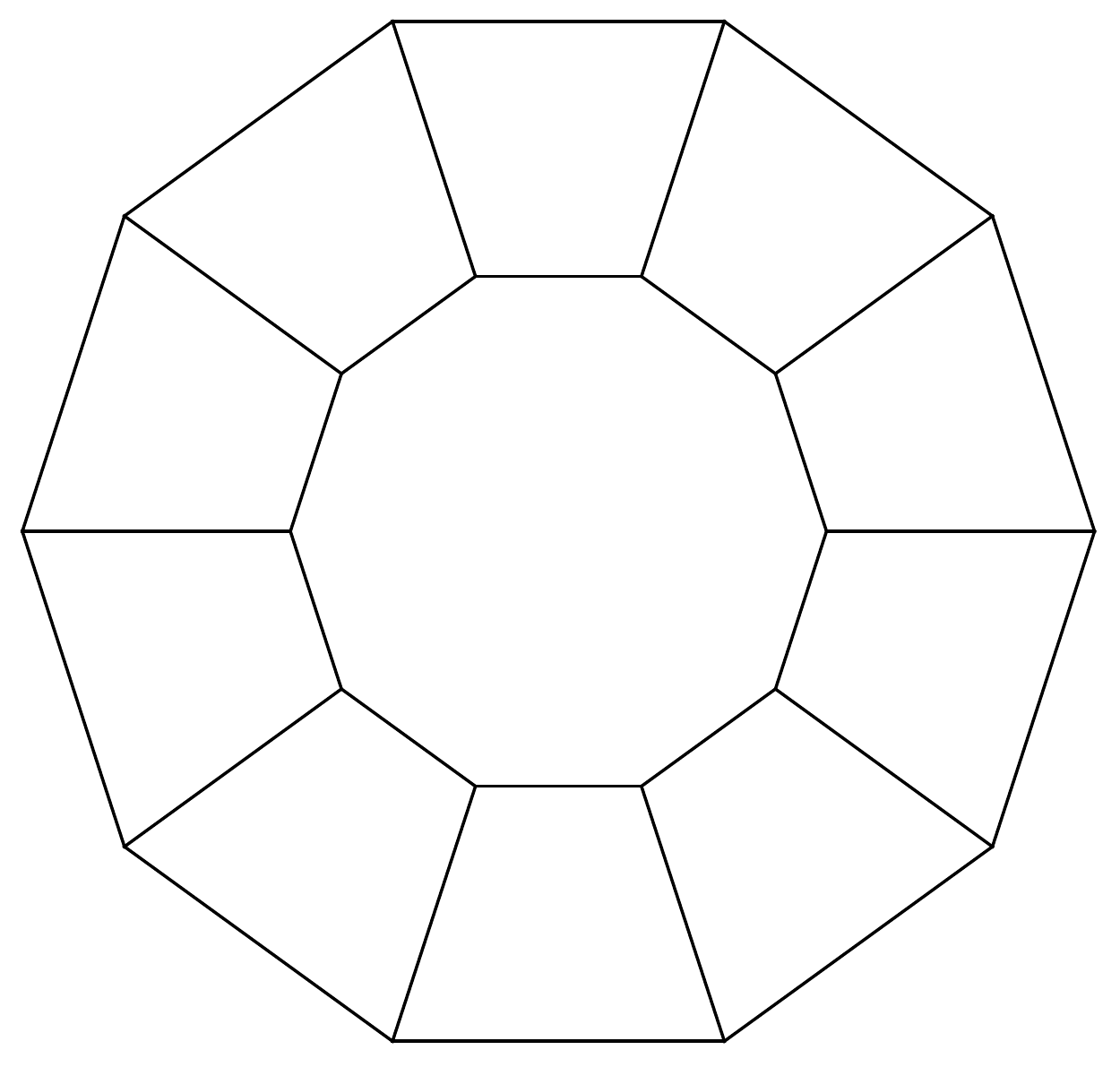}
\caption{}\label{fig:TenSpoke}
\end{figure}

According to Theorem~\ref{thm:main2}, $\reg(C^r(\widehat{\QC}))\le \max\{(r+1)(n+2),5(r+1)\}\le (r+1)(n+2)$ as long as $n\ge 3$.  However, according to computations for $r\le 3$ and $n\le 10$ in Macaulay2, $\reg(\QC)\le 3(r+1)$ regardless of what value $n$ takes.  It appears that having a facet $\sigma$ with many codimesion one facets may only significantly effect the regularity of $C^\alpha(\PC)$ if $\sigma\cap\partial\PC\neq\emptyset$, as in Example~\ref{ex:HighGen}.
\end{exm}

\begin{exm}\label{ex:Oct} Consider a regular octahedron $\Delta\subset\R^3$ triangulated by placing a centrally symmetric vertex, shown in Figure~\ref{fig:OctFig}.  In~\cite[Example~5.2]{Spect}, Schenck shows that $C^r(\Delta)$ is free, generated in degrees $r+1,2(r+1),$ and $3(r+1)$.  Thus the regularity bound for $C^r(\Delta)$ given by Corollary~\ref{cor:freegens} is tight.
\begin{figure}[htp]
\centering
\includegraphics[scale=.5]{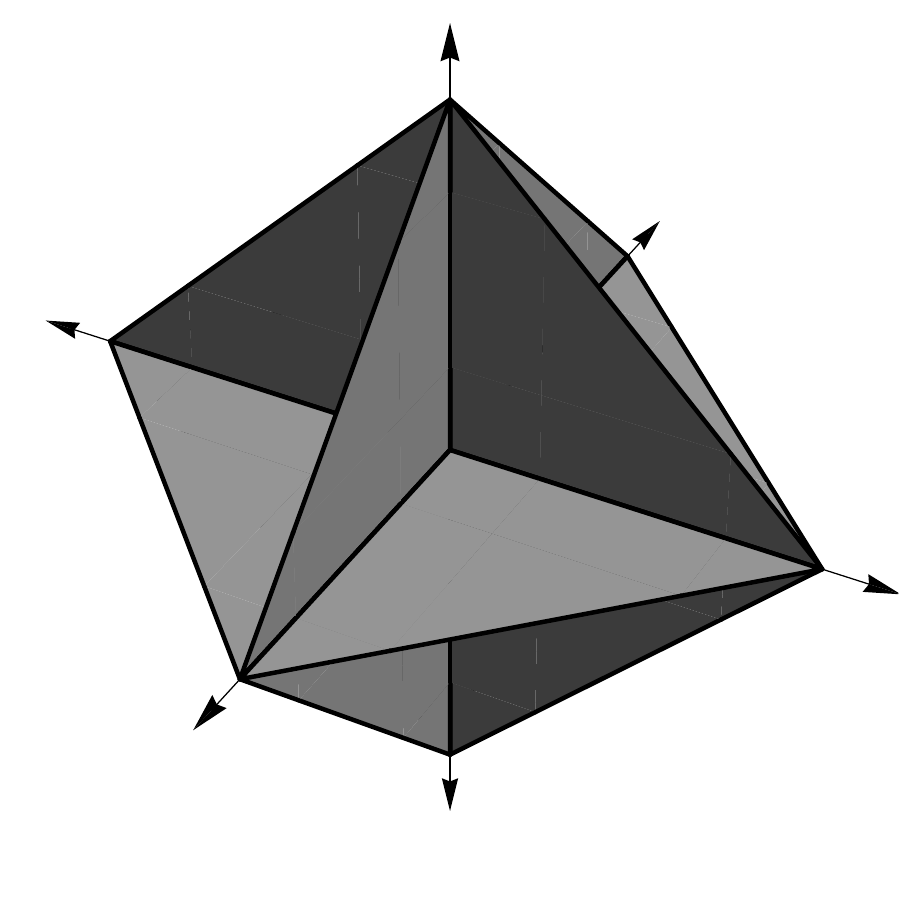}
\caption{Centrally Triangulated Octahedron}\label{fig:OctFig}
\end{figure}
Computations in Macaulay2 suggest that the regularity of $C^r(\Delta)$ stays at $3(r+1)$ for generic perturbations of the noncentral vertices.
\end{exm}

\section{Regularity Conjecture}\label{sec:conjectures}

Let $\PC\subset\R^n$ be a pure hereditary $n$-dimensional polytopal complex.  We end with a conjecture in the case of uniform smoothness, where $\alpha(\tau)=r$ for all $\tau\in\PC^0_{n-1}$ and $\alpha(\tau)=-1$ for $\tau\in(\partial\PC)_{n-1}$.  This conjecture is a slight refinement of ~\cite[Conjecture~5.6]{LS}.  For $\sigma\in\PC_n$, recall $|\partial^0(\sigma)|$ is the number of codimension one faces of $\sigma$ which are interior to the complex.  We will call a central complex $\PC$ \textit{complete} if the intersection of all facets of $\PC$ is an \textit{interior} face of $\PC$.

\begin{conj}\label{conj:c1}
Let $\PC\subset\R^{n+1}$ be a pure, central, hereditary $n$-dimensional polytopal complex.  Call a facet $\sigma\in\PC_{n+1}$ a boundary facet if it has a codimension one face $\tau\in\partial\PC$ so that $\tau$ contains the cone vertex. Set
\[
\begin{array}{ll}
F= & \max\{|\partial^0(\sigma)|:\sigma\in\PC_{n+1}\}\\
F_\partial= & \max\{|\partial^0(\sigma)|:\sigma\in\PC_{n+1}\mbox{ a boundary facet}\}
\end{array}
\]
Then,
\begin{enumerate}
\item If $\PC$ is central and complete, then $\reg (C^r(\PC))\le\reg (LS^{r,n-1}(\PC))\le F(r+1)$ and this bound is tight.
\item If $\PC$ is central but not complete, then $\reg (C^r(\PC))\le F_\partial(r+1)$.
\end{enumerate}
Furthermore, the bound is attained by free modules $C^r(\PC)$ in both cases.
\end{conj}

\begin{remark}
Example~\ref{ex:HighGen} shows that generators can be obtained in degree $F(r+1)$ for the complete central case and degree $F_\partial(r+1)$ in the non-complete central case, so these are the lowest possible regularity bounds that we can conjecture.
\end{remark}

\begin{remark}
If $C^r(\PC)$ is free, then $\reg (C^r(\PC))\le F(r+1)$ by Corollary~\ref{cor:freegens}.  Example~\ref{ex:Oct}, coupled with Theorem~\ref{thm:main3}, shows that Conjecture~\ref{conj:c1} is true in the complete, central, three dimensional, simplicial case.  If $\PC\subset\R^3$ is complete, central, and non-simplicial, then Conjecture~\ref{conj:c1} should be provable using the methods of \S~\ref{sec:SStar}.  The difficulty is in analyzing the ideal $K(\tau)$ from Lemma~\ref{lem:splineideal}.
\end{remark}

\begin{remark}
Conjecture~\ref{conj:c1} part (2) is a natural generalization of a conjecture of Schenck~\cite{CohVan}, that $\reg (C^r(\wDelta))\le 2(r+1)$ for $\Delta\subset\R^2$.  This is a highly nontrivial conjecture in the simplicial case; it implies, for instance, that $\wp(C^1(\wDelta))\le 2$.  To date, it is unknown whether $HP(C^1(\wDelta),3)=\dim C^1_3(\wDelta)$.  The difficulty of this problem is in large part due to the fact that non-local geometry plays an increasingly important role in low degree~\cite{APSr14,NonEx}.  Since our methods hinge on using the algebras $LS^{\alpha,k}(\PC)$, which are locally supported approximations to $C^\alpha(\PC)$, our approach will not be effective in proving Conjecture~\ref{conj:c1} part (2).
\end{remark}

\begin{remark}
In the non-simplicial case, Conjecture~\ref{conj:c1} part (2) appears to run contrary to the spirit of the regularity bounds we have proved in this paper, since no account is taken of interior facets of $\PC$, which may have many codimension one faces.  It is nevertheless consistent with Example~\ref{ex:HighGen}, where the minimal generator of high degree is supported on a boundary facet, and Example~\ref{ex:LargeFace}, where an interior facet with many codimension one faces appears to have no contribution to $\reg(C^r(\wPC))$.  An example of a polytopal complex $\PC$ with a minimal generator of high degree (relative to the number of codimension one faces of boundary facets), supported on interior facets, would be very interesting.
\end{remark}

\section{Acknowledgements}

I thank Hal Schenck for his guidance and patient listening.  I also thank Stefan Tohaneanu and Jimmy Shan for helpful conversations.  Macaulay2~\cite{M2} was indispensable for performing the computations in this paper.  Several of the images were generated using Mathematica.

\end{document}